\documentclass[11pt]{amsart}
\usepackage{amsmath, amssymb, comment, graphicx, url, amsthm, multicol, latexsym, enumerate, bbm, mathtools, physics, microtype, cite, xcolor, float}
\usepackage{units}
\usepackage[all,cmtip]{xy}
\usepackage{wrapfig}
\usepackage[hmargin=1in]{geometry} 
\usepackage{tikz}
\usepackage{tikzsymbols}
\usetikzlibrary{shapes.geometric, arrows}
\usepackage[cm]{sfmath}
\usepackage{hyperref}
\hypersetup{
    colorlinks=true,
    linkcolor=blue,
    citecolor=magenta,
    filecolor=magenta,      
    urlcolor=magenta,
}

\definecolor{andresblue}{rgb}{0,0.72,0.92}
\definecolor{andrespink}{rgb}{1,0,1}

\newtheorem{theorem}{Theorem}[section]
\newtheorem{proposition}[theorem]{Proposition}
\newtheorem{claim}[theorem]{Claim}

\newtheorem{lemma}[theorem]{Lemma}
\newtheorem{corollary}[theorem]{Corollary}

\theoremstyle{definition}
\newtheorem{remark}[theorem]{Remark}
\newtheorem{definition}[theorem]{Definition}

\newtheorem{example}[theorem]{Example}
\newtheorem{problem}[theorem]{Problem}
\newtheorem{question}[theorem]{Question}
\theoremstyle{remark}
\newtheorem{case}{Case}

\newcommand{\R}{\mathbb{R}}

\newcommand{\Z}{\mathbb{Z}}
\newcommand{\N}{\mathbb{N}}

\newcommand{\bv}{\mathbf{v}}

\newcommand{\bp}{\mathbf{p}}

\newcommand{\bx}{\mathbf{x}}

\newcommand{\by}{\mathbf{y}}
\newcommand{\be}{\mathbf{e}}

\newcommand{\Ehr}{\mathrm{Ehr}}

\newcommand{\defterm}[1]{\emph{#1}}

\DeclareMathOperator{\conv}{conv}
\DeclareMathOperator{\aff}{aff}

\DeclareMathOperator{\proj}{proj}

\DeclareMathOperator{\cone}{cone}

\makeatletter 
\newtheorem*{rep@theorem}{\rep@title}\newcommand{\newreptheorem}[2]{%
\newenvironment{rep#1}[1]{%
\def\rep@title{\bf #2 \ref{##1}}%
\begin{rep@theorem}}%
{\end{rep@theorem}}}
\makeatother
\newreptheorem{theorem}{Theorem}

\newtheorem*{rep@proposition}{\rep@title}\newcommand{\newrepproposition}[2]{%
\newenvironment{rep#1}[1]{%
\def\rep@title{\bf #2 \ref{##1}}%
\begin{rep@proposition}}%
{\end{rep@proposition}}}
\makeatother
\newreptheorem{proposition}{Proposition}

\usepackage[colorinlistoftodos]{todonotes}

\definecolor{munsell}{rgb}{0.0, 0.5, 0.69}

\begin{document}


\title{Stack-sorting simplices: geometry and lattice-point enumeration}

\author{Eon Lee}
\address{\scriptsize{School of Electrical Engineering and Computer Science, Gwangju Institute of Science and Technology}}
\email{\scriptsize{eonlee1125@gmail.com}}

\author{Carson Mitchell}
\address{\scriptsize{Department of Mathematics, University of Southern California}}
\email{\scriptsize{carsonrm@usc.edu}}

\author{Andr\'es R. Vindas-Mel\'endez}
\address{\scriptsize{Department of Mathematics, Harvey Mudd College}, \url{https://math.hmc.edu/arvm/}}
\email{\scriptsize{avindasmelendez@g.hmc.edu}}



\begin{abstract}
We initiate the study of subpolytopes of the permutahedron that arise as the convex hulls of 
stack-sorting on permutations. 
We primarily focus on $Ln1$ permutations, i.e., permutations of length $n$ whose penultimate and last entries are $n$ and $1$, respectively.
First, we present some enumerative results on $Ln1$ permutations.
Then we show that the polytopes that arise from stack-sorting on $Ln1$ permutations are simplices and proceed to study their geometry and lattice-point enumeration.
In addition, we pose questions and problems for further investigation. 
Particular focus is then taken on the $Ln1$ permutation $23\cdots n1$.
We show that the convex hull of all its iterations through the stack-sorting algorithm shares the same lattice-point enumerator as that of the $(n-1)$-dimensional unit cube and lecture-hall simplex.
Lastly, we detail some results on the real lattice-point enumerator for variations of the simplices arising from stack-sorting on the permutation $23\cdots n1$.
This then allows us to show that those simplices are Gorenstein of index $2$. 
\end{abstract}


\maketitle


\section{Introduction}
The study of sorting a permutation using stacks was first introduced by Knuth in the 1960s \cite{knuth68}.
Classically, the aim of the stack-sorting problem is to chronologically sort a permutation using a last-in/first-out algorithm.
In its simplest form, a stack is used to rearrange a permutation $\boldsymbol{\pi}={\pi}_1{\pi}_2\cdots {\pi}_{n-1}{\pi}_n$ as follows.  
The entries of $\boldsymbol{\pi}$ are pushed onto an originally empty stack and an output permutation is formed by popping elements from the stack.
We present the stack-sorting algorithm in more detail in the next section.
Stack-sorting has been extensively studied in many contexts and connected to different objects, including permutrees \cite{permutrees}, pattern avoidance \cite{sorting-avoiding}, enumerative combinatorial perspectives \cite{defant, 132-avoiding, West}, and computational complexity \cite{complexity}.

In this paper we aim to connect the stack-sorting algorithm with polyhedral geometry. 
A well-known polytope whose vertices are indexed by permutations in the
symmetric group $\mathfrak{S}_n$ is the permutahedron. 
The permutahedron $\Pi_n$ is an $(n-1)$-dimensional polytope defined as 
\[\Pi_n:=\conv{\{\;(\pi_1,\pi_2,\dots,\pi_n)\in \R^n \;\vert\; \pi_1\pi_2\cdots\pi_n \in \mathfrak{S}_n\;\}}.  \]
Above and in the rest of the paper, we make no distinction between a permutation and its corresponding point in
$\R^n$.
The permutahedron was first referenced in the literature in \cite{PHS} and its subpolytopes have been studied through different approaches, including pattern-avoidance \cite{DavisSagan} and group actions and slices \cite{ArdilaSchindlerVindas, ArdilaSupinaVindas, BrandenburgDeLoeraMeroni}.
Due to both the stack-sorting algorithm and the permutahedron being intimately connected to permutations, we ask the following motivating question.

\begin{question}\label{question1}
What are the subpolytopes of the permutahedron that arise as the convex hull of permutations of length $n$ obtained from iterations of the stack-sorting algorithm?    
\end{question}

Answers to Question \ref{question1} allow for the exploration of the combinatorial and geometrical structures of these subpolytopes.  
In particular, the enumeration of its $k$-faces, volumes, and their lattice-point enumeration. 
Additionally, it can provide geometric intuition into how the stack-sorting algorithm traverses the faces of the permutahedron.

For a permutation $\boldsymbol{\pi} \in \mathfrak{S}_n$, we define
\[\mathcal{S}^{\boldsymbol{\pi}}:=\{ s^0(\boldsymbol{\pi}), s^1(\boldsymbol{\pi}), s^2(\boldsymbol{\pi}), \ldots \}, \]
to be the set of output permutations from each iteration of the stack-sorting algorithm until it terminates at outputting the identity permutation.
The polyhedral geometry comes into play by taking the convex hull of $\mathcal{S}^{\boldsymbol{\pi}}$. 
Note that this polytope is a subpolytope of the permutahedron since the vertices of the permutahedron are exacly the $n!$ permutations of $\{1,2,\dots,n\}$, which are the possible permutation inputs for the algorithm.
There are plenty of permutation patterns that one can study and we focus our attention on a particular permutation pattern, which we now present.
We call a permutation $\boldsymbol{\pi} \in \mathfrak{S}_n$ an $Ln1$ permutation if it has the form $Ln1$, where $L$ is a permutation of $\lbrace 2, 3, \ldots, n-1 \rbrace$.
We denote the set of all permutations of that form as $\mathcal{L}^n$. 
Though we primarily focus on $Ln1$ permutations, we encourage the study of the polytopes that arise from stack-sorting on different permutation patterns, e.g., $231$-avoiding permutations.
Additionally, one may study the polytopes that arise from other sorting algorithms including those that involve repeated numbers in the permutation.

This paper is organized as follows. 
In Section \ref{sec:prelims}, we present preliminaries and relevant background on the stack-sorting algorithm, polyhedral geometry, and Ehrhart theory. 
Section \ref{sec:Ln1 perms} presents enumerative and structural results on $Ln1$ permutations.    
In Section \ref{sec:geometry}, we initiate the study of the geometry of $\conv( \mathcal{S}^{\boldsymbol{\pi}})$, for $\boldsymbol{\pi} \in \mathcal{L}^n$. 
We further concentrate on the enumerative geometric combinatorics of the polytopes arising from our set up when $\boldsymbol{\tau}_n:= 23\cdots(n-1)n1\in \mathcal{L}^n$.
The main results of this section are the following: 

\begin{repproposition}{corr:simpl}
For $\boldsymbol{\pi}\in \mathcal{L}^n$, $\conv(\mathcal{S}^{\boldsymbol{\pi}})$ forms an $(n-1)$-simplex in $\R^n$. 
\end{repproposition}

\begin{reptheorem}{thm:hollow}
For $\boldsymbol{\tau}_n\in \mathcal{L}^n$, the $\mathcal{S}^{\boldsymbol{\tau}_n}$ simplex $\triangle_n:=\conv(\mathcal{S}^{\boldsymbol{\tau}_n})$ is hollow. 
Specifically, all non-vertex integer points of $\triangle_n$ lie on the facet formed from the convex hull of $\mathcal{S}^{\boldsymbol{\tau}_n} \setminus \{12\cdots n\}$. 
\end{reptheorem}

Lastly, Section \ref{sec:ehrhart} deals with the Ehrhart theory of the $\mathcal{S}^{\boldsymbol{\tau}_n}$ simplices. 
The main result of this section is the following: 

\begin{reptheorem}{thm:integral_equivalence}
    The $\mathcal{S}^{\boldsymbol{\tau}_n}$ simplex $\triangle_{n}$ and $(n-1)$-dimensional lecture-hall simplex $P_{n-1}$ are integrally equivalent. 
    In particular, \[L_\Z(\triangle_{n}; t) = L_\Z(P_{n-1}; t)=(t+1)^{n-1}.\]
\end{reptheorem}

We conclude by exploring the following additional results, including developing a recursive relationship and proving the Gorenstein property for translates of the $\mathcal{S}^{\boldsymbol{\tau}_n}$ simplex.



\begin{reptheorem}{thm:sssrecurrence}
For all $\lambda \in \R_{\geq 0}$,
  \[
  L_\R(\triangle_{n+1}- \boldsymbol{\tau}_{n+1};\lambda) = \sum_{k=0}^{\lfloor n\lambda\rfloor}\ L_\R \left(\triangle_n-\boldsymbol{\tau}_n;\frac{k}{n}\right).
  \]
\end{reptheorem}

\begin{reptheorem}{thm:gorenstein}
  For all $\lambda \in \R_{\geq 0}$,
  \[
    L_\R(\triangle_n - \boldsymbol{\tau}_n ; \lambda) = L_\R((\triangle_n - \boldsymbol{\tau}_n)^{\circ}; \lambda+2).
  \]
  Hence, any integer translate of $\triangle_n - \boldsymbol{\tau}_n$, in particular $\triangle_n$, is Gorenstein of index $2$.
\end{reptheorem}

\section{Background \& Preliminaries}\label{sec:prelims}
\subsection{The stack-sorting algorithm}\label{subsec:stacks} \text{} 

The following algorithm for sorting an input sequence was first outlined in Donald Knuth's influential work, \emph{The Art of Computer Programming} \cite{knuth68}.

\begin{definition}\label{def:ssa}
Given a permutation $\boldsymbol{\pi}=\pi_1\pi_2\dots\pi_n$ on an ordered set, the \defterm{stack-sorting algorithm} is defined as follows: 
      \begin{enumerate}
          \item Initialize an empty \defterm{stack} $S=\lbrace \rbrace$ (i.e., a last-in, first-out collection of elements).
          \item For each input value $\pi_i$ of $\boldsymbol{\pi}$, starting with the first element, $\pi_1$: 
          \begin{itemize} 
          \item If the stack is non-empty and $\pi_i$ is greater than the \defterm{top} (most-recently added) element $\pi_{j}$ of the stack, where $j<i$, then \defterm{pop $\pi_j$ from the stack} (i.e., move $\pi_j$ to the output) and repeat.
          \item Otherwise, \defterm{push $\pi_i$ onto the stack} (move $\pi_i$ to the top of the stack). 
          \end{itemize} 
          \item Once every entry of the permutation has been pushed and there is no input left to consider, pop all remaining in the stack from last-in to first-in.
      \end{enumerate}
      
  \end{definition}

\noindent For a permutation $\boldsymbol{\pi}$, we denote the output of the stack-sorting algorithm as $s(\boldsymbol{\pi})$. 

\begin{figure}[!ht]
    \centering
    \begin{tikzpicture}[scale=0.5]
        \node at (-1,0.5) {$(i)$};
        \draw[thick] (0,0)--(3,0)--(3,-3)--(4,-3)--(4,0)--(7,0);
        \draw (4,0) rectangle node {$2$} (5,1);
        \draw (5,0) rectangle node {$1$} (6,1);
        \draw (6,0) rectangle node {$3$} (7,1);
        \draw[thick, ->] (4,0.5) arc (90:180:0.5cm);
    \end{tikzpicture}
    \qquad
    \begin{tikzpicture}[scale=0.50]
        \node at (-1,0.5) {$(ii)$};
        \draw[thick] (0,0)--(3,0)--(3,-3)--(4,-3)--(4,0)--(7,0);
        \draw (3,-3) rectangle node {$2$} (4,-2);
        \draw (5,0) rectangle node {$1$} (6,1);
        \draw (6,0) rectangle node {$3$} (7,1);
        \draw[thick, ->] (5,0.5) arc (90:180:1.5cm);
    \end{tikzpicture}

    \begin{tikzpicture}[scale=0.50]
        \node at (-1,0.5) {$(iii)$};
        \draw[thick] (0,0)--(3,0)--(3,-3)--(4,-3)--(4,0)--(7,0);
        \draw (3,-3) rectangle node {$2$} (4,-2);
        \draw (3,-2) rectangle node {$1$} (4,-1);
        \draw (6,0) rectangle node {$3$} (7,1);
        \draw[thick, ->] (3.5,-1) arc (0:90:1.5cm);
    \end{tikzpicture}
    \qquad
    \begin{tikzpicture}[scale=0.5]
        \node at (-1,0.5) {$(iv)$};
        \draw[thick] (0,0)--(3,0)--(3,-3)--(4,-3)--(4,0)--(7,0);
        \draw (3,-3) rectangle node {$2$} (4,-2);
        \draw (0,0) rectangle node {$1$} (1,1);
        \draw (6,0) rectangle node {$3$} (7,1);
        \draw[thick, ->] (3.5,-2) arc (0:30:5cm);
    \end{tikzpicture}
    \caption{A visualization of step $2$ of the stack-sorting algorithm as presented in Example~\ref{ex:ssa}.}
\end{figure}

  \begin{example}\label{ex:ssa}
      The algorithm would sort the permutation $\boldsymbol{\pi} = 213$ as follows:
      \begin{enumerate}
          \item We begin with an empty stack $S=\lbrace \rbrace$.
          \item The first entry in the permutation is $\pi_1=2$.
          The stack is empty; thus, we push $2$ onto the stack and obtain $S=\lbrace 2 \rbrace$.
          \begin{itemize}
          \item The next entry is $1$.
          The stack is non-empty but $1$ is not greater the the top of the stack, $2$. Therefore, we push $1$ onto the stack and obtain $S = \lbrace 2,1 \rbrace$.
          \item The last permutation entry is $3$.
          The stack is non-empty and $3$ is greater than the top, $1$. Thus, we pop $1$ from the stack and obtain the first element of $s(\boldsymbol{\pi})$ to be $1$. We repeat and the exact same thing occurs with $2$. We deduce that $s(\boldsymbol{\pi})$ begins with $12$.
          \end{itemize}
      \item We push $3$ onto an empty stack and then pop it to obtain that $s(\boldsymbol{\pi}) = 123$. \\
      
      \hfill $\diamondsuit$
      \end{enumerate}
  \end{example}

  \begin{definition}\label{def:tss}
  A permutation $\boldsymbol{\pi}\in\mathfrak{S}_n$ is \defterm{$t$-stack-sortable} if $t$ iterations of the stack-sorting algorithm yields the identity permutation, i.e., $s^t(\boldsymbol{\pi}) = 12\cdots n$. 
  Furthermore, $\boldsymbol{\pi}$ is \emph{exactly} $t$-stack-sortable if $s^t(\boldsymbol{\pi}) = 12\cdots n$ and $m < t$ implies $s^m(\boldsymbol{\pi}) \neq 12\cdots n$.
  \end{definition}
  
  The identity permutation itself is defined to be $0$-stack-sortable, and $s(12\cdots n) = 12\cdots n$ for all $n$. 
  All permutations on $[n]$ will reach the identity after at most $n-1$ iterations of the algorithm.
  If $\boldsymbol{\pi}$ is exactly $(n-1)$-stack-sortable, then $\boldsymbol{\pi}$ is  said to be \defterm{maximal} with respect to the algorithm.
  Thus, revisiting Example~\ref{ex:ssa}, $\boldsymbol{\pi}=213$ is exactly $1$-stack-sortable.
  It is also $t$-stack-sortable for any $t \geq 1$, and thus is not maximal.

\subsection{Polyhedral geometry \& Ehrhart theory} \label{subsec:geom}\text{} 

In this paper, we investigate the properties of a family of \emph{convex polytopes}, in particular, a family of lattice simplices. 
We present some preliminaries on polyhedral geometry and Ehrhart theory, i.e., the study of lattice-points in dilations of polytopes.

A \defterm{convex polytope} $P$ is the \defterm{convex hull} of a set of points $C = \{ \bx_1, \ldots, \bx_k \}$, that is, 
\begin{equation}\label{eqn: convex hull} P =\conv(C) := \left\lbrace \sum_{i=1}^k \lambda_i \bx_i: \lambda_i \in \R_{\geq 0} \text{ and } \sum_{i=1}^k \lambda_i = 1 \right\rbrace.
\end{equation}
Any expression of the form $\sum_{i=1}^k \lambda_i \bx_i$, with the same conditions as imposed on the $\lambda_i$ in (\ref{eqn: convex hull}), is called a \defterm{convex combination} of the points $\bx_i$.
The convex hull of $C$ is the set of all convex combinations of points in $C$. 

We define a \defterm{hyperplane} in $\R^n$ as an $(n-1)$-dimensional affine subspace determined by the solution space of a linear equation \[a_0 + a_1x_1 + \cdots + a_nx_n = 0,\] where each $a_i \in \R$ and at least one of $a_1$ through $a_n$ is nonzero. 
Any hyperplane of $\R^n$ partitions $\R^n$ into two parts called \defterm{halfspaces}; namely, without loss of generality, \[
    a_0 + a_1x_1 + \cdots + a_nx_n \geq 0 \quad \text{ and } \quad a_0 + a_1x_1 + \cdots + a_nx_n < 0.\]
A polytope $P$ can also be defined as the bounded intersection of finitely many \defterm{halfspaces}.

A polytope $P$ is contained in the \emph{affine hull} of $C$, denoted by $\aff(C)$, which is similar to \eqref{eqn: convex hull}, but with the looser condition that $\lambda_i\in \R$. 
We can characterize $\aff(C)$ as the smallest affine subspace containing $\conv(C)$.
Moreover, $C =\{\bx_1,\ldots,\bx_k\}$ is \emph{affinely independent} if $\aff(C)$ is a $(k-1)$-dimensional space.
Equivalently, $C$ is affinely independent if and only if vectors, \[\bx_1 - \bx_i,\dots,\bx_{i-1} - \bx_i,\bx_{i+1} -\bx_i, \dots, \bx_k-\bx_i\ \in \R^n,\] 
are linearly independent in $\R^n$ for any $i \in [k]$. 

The \defterm{dimension} of a polytope $P$, denoted $\dim(P)$, is the dimension of its affine hull.
We say a polytope $P\subset \R^n$ is \emph{full-dimensional} if $\dim(P)=n$.
The convex hull of $n+1$ affinely independent points forms a \defterm{simplex} of dimension $n$.

Changing all inequalities to strict inequalities in the minimal halfspace description of $P$ yields the \defterm{interior} of $P$, denoted $P^{\circ}$. 
Changing all inequalities to equalities in the minimal halfspace description yields the \defterm{boundary} of $P$, denoted $\partial P$.
Note that $P = P^{\circ} \uplus \partial P$, where $\uplus$ denotes a disjoint union.
A polytope is \defterm{hollow} if all its integer points lie on its boundary or, equivalently, it has no integer points in its interior.
A polytope is said to be a \defterm{lattice polytope} if all of its vertices have integral coordinates.
One important lattice polytope is the permutahedron $\Pi_n$, i.e., the $(n-1)$-dimensional polytope defined as the convex hull of all permutations of the coordinates of $(1,2,\dots,n)\in\R^n$.
\begin{figure}[h!]
\centering
\begin{tikzpicture}%
	[scale=.7,
	edge/.style={color=black},
	facet/.style={fill=andresblue,fill opacity=0.700000}]
%
%

\draw[edge] (.86603, 1.50000) -- (0.00000, 2.00000);

\draw[andrespink,fill=andrespink] (0.86603, 1.50000) circle (2pt);
\node[anchor=west] at (0.86603, 1.50000) {\color{blue}$21$};
\coordinate (0.00000, 2.00000) at (0.00000, 2.00000);
\draw[andrespink,fill=andrespink] (0.00000, 2.00000) circle (2pt);
\node[anchor=south] at (0.00000, 2.00000) {\color{blue}$12$};




\end{tikzpicture}
\qquad
\begin{tikzpicture}%
	[scale=.7,
	edge/.style={color=black},
	facet/.style={fill=andresblue,fill opacity=0.300000}]
%
%
\coordinate (0.86603, -0.50000) at (0.86603, -0.50000);
\draw[andrespink,fill=andrespink] (0.86603, -0.50000) circle (2pt);
\node[anchor=west] at (0.86603, -0.50000) {\color{blue}$321$};
\coordinate (0.86603, 0.50000) at (0.86603, 0.50000);
\draw[andrespink,fill=andrespink] (0.86603, 0.50000) circle (2pt);
\node[anchor=west] at (0.86603, 0.50000) {\color{blue}$231$};
\coordinate (0.00000, -1.00000) at (0.00000, -1.00000);
\draw[andrespink,fill=andrespink]  (0.00000, -1.00000) circle (2pt);
\node[anchor=north] at (0.00000, -1.00000) {\color{blue}$312$};
\coordinate (0.00000, 1.00000) at (0.00000, 1.00000);
\draw[andrespink,fill=andrespink] (0.00000, 1.00000) circle (2pt);
\node[anchor=south] at (0.00000, 1.00000) {\color{blue}$132$};
\coordinate (-0.86603, -0.50000) at (-0.86603, -0.50000);
\draw[andrespink,fill=andrespink] (-0.86603, -0.50000) circle (2pt);
\node[anchor=east] at (-0.86603, -0.50000) {\color{blue}$213$};
\coordinate (-0.86603, 0.50000) at (-0.86603, 0.50000);
\draw[andrespink,fill=andrespink] (-0.86603, 0.50000) circle (2pt);
\node[anchor=east] at (-0.86603, 0.50000) {\color{blue}$123$};

\fill[facet] (-0.86603, 0.50000) -- (0.00000, 1.00000) -- (0.86603, 0.50000) -- (0.86603, -0.50000) -- (0.00000, -1.00000) -- (-0.86603, -0.50000) -- cycle {};

\coordinate (0.00000, 0.00000) at (0.00000, 1.00000);
\draw[andrespink,fill=andrespink] (0.00000, 0.00000) circle (2pt);
\node[anchor=south] at (0.00000, 0.00000) {\color{blue}$$};
\draw[edge] (0.86603, -0.50000) -- (0.86603, 0.50000);
\draw[edge] (0.86603, -0.50000) -- (0.00000, -1.00000);
\draw[edge] (0.86603, 0.50000) -- (0.00000, 1.00000);
\draw[edge] (0.00000, -1.00000) -- (-0.86603, -0.50000);
\draw[edge] (0.00000, 1.00000) -- (-0.86603, 0.50000);
\draw[edge] (-0.86603, -0.50000) -- (-0.86603, 0.50000);

\coordinate (0.86603, -0.50000) at (0.86603, -0.50000);
\draw[andrespink,fill=andrespink] (0.86603, -0.50000) circle (2pt);
\node[anchor=west] at (0.86603, -0.50000) {\color{blue}$321$};
\coordinate (0.86603, 0.50000) at (0.86603, 0.50000);
\draw[andrespink,fill=andrespink] (0.86603, 0.50000) circle (2pt);
\node[anchor=west] at (0.86603, 0.50000) {\color{blue}$231$};
\coordinate (0.00000, -1.00000) at (0.00000, -1.00000);
\draw[andrespink,fill=andrespink]  (0.00000, -1.00000) circle (2pt);
\node[anchor=north] at (0.00000, -1.00000) {\color{blue}$312$};
\coordinate (0.00000, 1.00000) at (0.00000, 1.00000);
\draw[andrespink,fill=andrespink] (0.00000, 1.00000) circle (2pt);
\node[anchor=south] at (0.00000, 1.00000) {\color{blue}$132$};
\coordinate (-0.86603, -0.50000) at (-0.86603, -0.50000);
\draw[andrespink,fill=andrespink] (-0.86603, -0.50000) circle (2pt);
\node[anchor=east] at (-0.86603, -0.50000) {\color{blue}$213$};
\coordinate (-0.86603, 0.50000) at (-0.86603, 0.50000);
\draw[andrespink,fill=andrespink] (-0.86603, 0.50000) circle (2pt);
\node[anchor=east] at (-0.86603, 0.50000) {\color{blue}$123$};

\end{tikzpicture}
\qquad
\begin{tikzpicture}%
	[scale=.5,
	x={(0.767968cm, 0.559570cm)},
	y={(-0.407418cm, 0.802202cm)},
	z={(0.494203cm, -0.208215cm)},
	back/.style={loosely dotted, thick},
	edge/.style={color=black, thick},
	facet/.style={fill=andresblue,fill opacity=0.3}]
%
%
\coordinate (-0.70711, -1.22474, -1.73205) at (-0.70711, -1.22474, -1.73205);
\coordinate (-0.70711, -2.04124, -0.57735) at (-0.70711, -2.04124, -0.57735);
\coordinate (-1.41421, 0.00000, -1.73205) at (-1.41421, 0.00000, -1.73205);
\coordinate (-1.41421, -1.63299, 0.57735) at (-1.41421, -1.63299, 0.57735);
\coordinate (-2.12132, 0.40825, -0.57735) at (-2.12132, 0.40825, -0.57735);
\coordinate (-2.12132, -0.40825, 0.57735) at (-2.12132, -0.40825, 0.57735);
\coordinate (0.70711, -1.22474, -1.73205) at (0.70711, -1.22474, -1.73205);
\coordinate (0.70711, -2.04124, -0.57735) at (0.70711, -2.04124, -0.57735);
\coordinate (-0.70711, 1.22474, -1.73205) at (-0.70711, 1.22474, -1.73205);
\coordinate (-0.70711, -1.22474, 1.73205) at (-0.70711, -1.22474, 1.73205);
\coordinate (-1.41421, 1.63299, -0.57735) at (-1.41421, 1.63299, -0.57735);
\coordinate (-1.41421, 0.00000, 1.73205) at (-1.41421, 0.00000, 1.73205);
\coordinate (1.41421, 0.00000, -1.73205) at (1.41421, 0.00000, -1.73205);
\coordinate (1.41421, -1.63299, 0.57735) at (1.41421, -1.63299, 0.57735);
\coordinate (0.70711, 1.22474, -1.73205) at (0.70711, 1.22474, -1.73205);
\coordinate (0.70711, -1.22474, 1.73205) at (0.70711, -1.22474, 1.73205);
\coordinate (-0.70711, 2.04124, 0.57735) at (-0.70711, 2.04124, 0.57735);
\coordinate (-0.70711, 1.22474, 1.73205) at (-0.70711, 1.22474, 1.73205);
\coordinate (2.12132, 0.40825, -0.57735) at (2.12132, 0.40825, -0.57735);
\coordinate (2.12132, -0.40825, 0.57735) at (2.12132, -0.40825, 0.57735);
\coordinate (1.41421, 1.63299, -0.57735) at (1.41421, 1.63299, -0.57735);
\coordinate (1.41421, 0.00000, 1.73205) at (1.41421, 0.00000, 1.73205);
\coordinate (0.70711, 2.04124, 0.57735) at (0.70711, 2.04124, 0.57735);
\coordinate (0.70711, 1.22474, 1.73205) at (0.70711, 1.22474, 1.73205);
\draw[edge,back] (-0.70711, -1.22474, -1.73205) -- (-0.70711, -2.04124, -0.57735);
\draw[edge,back] (-0.70711, -1.22474, -1.73205) -- (-1.41421, 0.00000, -1.73205);
\draw[edge,back] (-0.70711, -1.22474, -1.73205) -- (0.70711, -1.22474, -1.73205);
\draw[edge,back] (-0.70711, -2.04124, -0.57735) -- (-1.41421, -1.63299, 0.57735);
\draw[edge,back] (-0.70711, -2.04124, -0.57735) -- (0.70711, -2.04124, -0.57735);
\draw[edge,back] (-1.41421, 0.00000, -1.73205) -- (-2.12132, 0.40825, -0.57735);
\draw[edge,back] (-1.41421, 0.00000, -1.73205) -- (-0.70711, 1.22474, -1.73205);
\draw[edge,back] (0.70711, -1.22474, -1.73205) -- (0.70711, -2.04124, -0.57735);
\draw[edge,back] (0.70711, -1.22474, -1.73205) -- (1.41421, 0.00000, -1.73205);
\draw[edge,back] (0.70711, -2.04124, -0.57735) -- (1.41421, -1.63299, 0.57735);
\draw[edge,back] (1.41421, 0.00000, -1.73205) -- (0.70711, 1.22474, -1.73205);
\draw[edge,back] (1.41421, 0.00000, -1.73205) -- (2.12132, 0.40825, -0.57735);
\fill[facet] (1.41421, 0.00000, 1.73205) -- (0.70711, -1.22474, 1.73205) -- (1.41421, -1.63299, 0.57735) -- (2.12132, -0.40825, 0.57735) -- cycle {};
\fill[facet] (0.70711, 1.22474, 1.73205) -- (-0.70711, 1.22474, 1.73205) -- (-1.41421, 0.00000, 1.73205) -- (-0.70711, -1.22474, 1.73205) -- (0.70711, -1.22474, 1.73205) -- (1.41421, 0.00000, 1.73205) -- cycle {};
\fill[facet] (-1.41421, 0.00000, 1.73205) -- (-2.12132, -0.40825, 0.57735) -- (-1.41421, -1.63299, 0.57735) -- (-0.70711, -1.22474, 1.73205) -- cycle {};
\fill[facet] (-0.70711, 1.22474, 1.73205) -- (-1.41421, 0.00000, 1.73205) -- (-2.12132, -0.40825, 0.57735) -- (-2.12132, 0.40825, -0.57735) -- (-1.41421, 1.63299, -0.57735) -- (-0.70711, 2.04124, 0.57735) -- cycle {};
\fill[facet] (0.70711, 1.22474, 1.73205) -- (-0.70711, 1.22474, 1.73205) -- (-0.70711, 2.04124, 0.57735) -- (0.70711, 2.04124, 0.57735) -- cycle {};
\fill[facet] (0.70711, 2.04124, 0.57735) -- (-0.70711, 2.04124, 0.57735) -- (-1.41421, 1.63299, -0.57735) -- (-0.70711, 1.22474, -1.73205) -- (0.70711, 1.22474, -1.73205) -- (1.41421, 1.63299, -0.57735) -- cycle {};
\fill[facet] (0.70711, 1.22474, 1.73205) -- (1.41421, 0.00000, 1.73205) -- (2.12132, -0.40825, 0.57735) -- (2.12132, 0.40825, -0.57735) -- (1.41421, 1.63299, -0.57735) -- (0.70711, 2.04124, 0.57735) -- cycle {};
\draw[edge] (-1.41421, -1.63299, 0.57735) -- (-2.12132, -0.40825, 0.57735);
\draw[edge] (-1.41421, -1.63299, 0.57735) -- (-0.70711, -1.22474, 1.73205);
\draw[edge] (-2.12132, 0.40825, -0.57735) -- (-2.12132, -0.40825, 0.57735);
\draw[edge] (-2.12132, 0.40825, -0.57735) -- (-1.41421, 1.63299, -0.57735);
\draw[edge] (-2.12132, -0.40825, 0.57735) -- (-1.41421, 0.00000, 1.73205);
\draw[edge] (-0.70711, 1.22474, -1.73205) -- (-1.41421, 1.63299, -0.57735);
\draw[edge] (-0.70711, 1.22474, -1.73205) -- (0.70711, 1.22474, -1.73205);
\draw[edge] (-0.70711, -1.22474, 1.73205) -- (-1.41421, 0.00000, 1.73205);
\draw[edge] (-0.70711, -1.22474, 1.73205) -- (0.70711, -1.22474, 1.73205);
\draw[edge] (-1.41421, 1.63299, -0.57735) -- (-0.70711, 2.04124, 0.57735);
\draw[edge] (-1.41421, 0.00000, 1.73205) -- (-0.70711, 1.22474, 1.73205);
\draw[edge] (1.41421, -1.63299, 0.57735) -- (0.70711, -1.22474, 1.73205);
\draw[edge] (1.41421, -1.63299, 0.57735) -- (2.12132, -0.40825, 0.57735);
\draw[edge] (0.70711, 1.22474, -1.73205) -- (1.41421, 1.63299, -0.57735);
\draw[edge] (0.70711, -1.22474, 1.73205) -- (1.41421, 0.00000, 1.73205);
\draw[edge] (-0.70711, 2.04124, 0.57735) -- (-0.70711, 1.22474, 1.73205);
\draw[edge] (-0.70711, 2.04124, 0.57735) -- (0.70711, 2.04124, 0.57735);
\draw[edge] (-0.70711, 1.22474, 1.73205) -- (0.70711, 1.22474, 1.73205);
\draw[edge] (2.12132, 0.40825, -0.57735) -- (2.12132, -0.40825, 0.57735);
\draw[edge] (2.12132, 0.40825, -0.57735) -- (1.41421, 1.63299, -0.57735);
\draw[edge] (2.12132, -0.40825, 0.57735) -- (1.41421, 0.00000, 1.73205);
\draw[edge] (1.41421, 1.63299, -0.57735) -- (0.70711, 2.04124, 0.57735);
\draw[edge] (1.41421, 0.00000, 1.73205) -- (0.70711, 1.22474, 1.73205);
\draw[edge] (0.70711, 2.04124, 0.57735) -- (0.70711, 1.22474, 1.73205);
\end{tikzpicture}
 \caption{Left to Right: $\Pi_2\subset\R^2$, $\Pi_3\subset\R^3$, and $\Pi_4\subset\R^4$.}
    \label{fig:permutahedra}
\end{figure}

The \defterm{Ehrhart function} (or \defterm{lattice-point enumerator} or \defterm{discrete volume}) of any polytope $P$ is defined as
\begin{equation*}
    L_\mathbb{Z}(P;t): = |tP\cap\Z^n|,
\end{equation*}
where $tP=\lbrace t\bx:\ \bx \in P\rbrace$ and $t$ is an integer.
When $P$ is a lattice polytope, the Ehrhart function is a polynomial in $t$ of degree equal to the dimension of $P$, which is refered to as the \defterm{Ehrhart polynomial} \cite{Ehrhart}.
We can encode the information of an Ehrhart polynomial in a generating series to obtain the \defterm{Ehrhart series} of a polytope:
\begin{equation*}
    \Ehr_\mathbb{Z}(P;z) := 1 + \sum_{t \in \Z_{>0}} L_\Z(P;t) z^n = \frac{h_{\Z}^{*}(P;z)}{(1-z)^{n+1}},
\end{equation*}
where $n$ is the dimension of $P$ and $h^*_\Z(P;z) = 1 + h^*_1 z + \cdots + h^*_n z^n$ is a polynomial in $z$ of degree at most $n$ called the \defterm{$h^*$-polynomial} of $P$. 
The coefficients $h^*_i$ of this polynomial are nonnegative integers and have relevant combinatorial or geometric interpretations.
For example, the sum of all $h^*_i$ is the normalized volume of the polytope and $h^*_n$ equals the number of interior lattice points of the polytope.
The interested reader can consult \cite{br15} for proofs of these examples and more information.

A fundamental result in Ehrhart theory is known as \emph{Ehrhart-Macdonald reciprocity} \cite{Ehrhart, Macdonald}:
for any convex lattice polytope $P$ of dimension $n$:
\begin{equation*}
    L_\Z(P;-t)=(-1)^nL_\Z(P^{\circ};t).
\end{equation*}
This not only provides a relationship between a polytope and its interior, but also provides an interpretation for negative values of $t$.
Evaluating an Ehrhart polynomial at $-t$ gives the lattice count for the $t^{\text{th}}$ dilate of its interior up to a sign.

A lattice polytope $P \subset \R^d$ is said to be \defterm{Gorenstein of index $k$} if there exists a positive integer $k$ such that \begin{equation}\label{eq:gorenstein}
    L_\Z(P^\circ; k-1) = 0,\quad L_\Z(P^\circ; k) =1,\quad \text{ and } \quad L_\Z(P^\circ; t) = L_\Z(P; t-k),
\end{equation} for all integers $t > k$ \cite{br15}.

More generally, for any convex polytope $P \subseteq \mathbb{R}^n$ (i.e., $P$ has any vertices in $\R^n$), the real Ehrhart counting function for all dilates $\lambda \in \R_{\geq 0}$ is defined as:
\begin{equation*}
    L_\mathbb{R}(P;\lambda): = |\lambda P\cap\Z^n|.
\end{equation*}
For more on Ehrhart theory, one can consult \cite{br15} for an in-depth treatment of the material. 
A larger survey on the combinatorics of polytopes can be found in \cite{ziegler95}.
Additionally, one can learn more about rational Ehrhart theory from \cite{BeckEliaRehberg}.

\begin{example}\label{ex:unitcube}
\begin{figure}[!ht]
    \centering
    \begin{tikzpicture}[scale=.75]
      \draw[thick, ->] (0,0) -- (0,4);
      \draw[thick, ->] (0,0) -- (4,0);
      
     \draw (0,0) rectangle (3,3);
      \fill[nearly transparent, andresblue] (0,0) rectangle (3,3);

    \draw (0,0) rectangle (2,2);
      \fill[nearly transparent, blue] (0,0) rectangle (2,2);
      
      \draw (0,0) rectangle (1,1);
      \fill[nearly transparent, andresblue] (0,0) rectangle (1,1);

      \node at (0.5,0.5) {$\square^2$};
      \node at (1,1.5) {$2\square^2$};
      \node at (1.5,2.5) {$3\square^2$};

    \foreach \i in {0,...,3}
      \foreach \j in {0,...,3}{
        \draw[fill = andrespink] (\i,\j) circle(3pt);
      };
    \end{tikzpicture}
    \caption{The first three integral dilates of the two-dimensional unit cube $\square^2$.}
    \label{fig:my_label}
\end{figure}

We define the \defterm{$n$-dimensional unit cube} as the convex hull over all binary strings of length $n$, i.e., the binary strings form the vertex set.
The two-dimensional unit cube $\square^2$ contains $4$ points in its first dilate, $9$ in its second, $16$ in its third, etc. 
Without much difficulty, one can show directly or inductively that: $$L_{\Z}(\square^2; t) = \left|t\square^2 \cap \Z^2 \right| =(t+1)^2.$$
Analogous argumentation can be used to show $L_{\Z}(\square^3; t) = (t+1)^3$; and, in fact, that \[L_{\Z}(\square^n; t) = (t+1)^n. \]
We can encode the lattice-point enumerator in a generating function to obtain the Ehrhart series of $\square^n$: 
\[\sum_{t \in \Z_{\geq 0}} L_{\Z}(\square^n; t) z^t = \frac{\sum_{k=0}^{n-1} A(n,k) z^k}{(1-z)^{n-1}}, \]
where $A(n,k)$ denotes the Eulerian number, i.e., the number of permutations of $1$ through $n$ having exactly $k$ descents.
Further, by reciprocity, we have:
    \[L_\Z(\square^n; -t) = (1-t)^n = (-1)^n(t-1)^n = (-1)^n L_\Z({\square^n}^\circ; t),\]
and can conclude the interior of the $t$-dilate has exactly as many lattice points as the entirety of the $(t-2)$-dilate. 
For example, the interior of $3\square^2$ has $4$ lattice points, the same number as all of $\square^2$. 
We also notice that the interior of the first dilate of $\square^n$ is empty and the interior of the second dilate has exactly one integer point. 
These are the conditions from (\ref{eq:gorenstein}) for $k=2$, which implies $\square^n$ is Gorenstein of index 2, an important property that will appear again later in the paper. \hfill $\diamondsuit$
\end{example}

\subsection{Integral and unimodular equivalence}\label{subsec:equivalence}
\text{}

We recall the definition of a unimodular transformation.

\begin{definition}
    A \defterm{unimodular transformation} in $\R^n$ is a linear transformation $U$, i.e., a $n$ by $n$ matrix, with coefficients in $\Z$ such that $\det(U) = \pm 1$.
\end{definition}

The following result from \cite{hk10} provides a sufficient, but not necessary condition, for determining whether two polytopes have the same lattice-point count. 

\begin{proposition}
    If a linear transformation on a lattice polytope $P$ is unimodular, then it preserves the lattice.
\end{proposition}

In other words, the resulting polytope will have the same integer point count as $P$. However, not all polytopes that have equivalent lattice counts have a unimodular transformation between them. For example, if the two polytopes have square matrix representations (vertices as columns of a matrix), they must have the same determinant up to a sign for a unimodular transformation to exist between them. 
We now have an effective method to show two $n$-polytopes are integrally equivalent: we can search for a unimodular transformation from the lattice points of one polytope to the other. 

\begin{definition}\label{def:integralequiv}
Two lattice polytopes $P\subset \R^m$ and $Q\subset \R^n$ are \emph{integrally equivalent} if there exists an affine transformation $\varphi: \R^m \rightarrow \R^n$ whose restriction to $P$ preserves the lattice. 
In other words, $\varphi$ is a bijection \[P\cap \Z^m \longleftrightarrow Q\cap \Z^n.\]
In particular, $L_\Z(P; t) = L_\Z(Q; t)$ for all $t \in \N$.
\end{definition}

\subsection{Lecture-hall simplices}\label{subsec:lecture}
\text{} 

We begin this subsection by recalling the definition of lecture-hall partitions, which were studied in \cite{BousquetEriksson, CorteelLeeSavage}.
They were further studied in the context of lecture-hall simplices in \cite{BrightSavage}.

\begin{definition}
    A \defterm{lecture-hall partition} of length $n$ is a sequence $\lbrace \alpha_1, \ldots, \alpha_n \rbrace \in \Z^n$ such that \[ 0 \leq \alpha_1 \leq \frac{\alpha_2}{2} \leq \cdots \leq \frac{\alpha_n}{n}.\]
    We can construct the \defterm{lecture-hall simplex} of dimension $n$ as \[P_n := \conv\left\lbrace \boldsymbol{\alpha} \in \Z^n: 0 \leq \alpha_1 \leq \frac{\alpha_2}{2} \leq \cdots \leq \frac{\alpha_n}{n} \leq 1 \right\rbrace.\]
\end{definition}

Dilations of $P_n$ yield the result $tP_n = \conv\left\lbrace \boldsymbol{\alpha} \in \Z^n: 0 \leq \alpha_1 \leq \frac{\alpha_2}{2} \leq \cdots \leq \frac{\alpha_n}{n} \leq t \right\rbrace$.
We also have that $P_n \subset 2P_n \subset \cdots \subset tP_n$ because $(0,\ldots,0) \in P_n$.
This fixes a vertex and thus avoids any translation of the polytope away from previous dilates, an observation that will later become useful.

\begin{example}
    For $n=4$, we must have $0 \leq \alpha_1 \leq \cdots \leq \frac{\alpha_4}{4} \leq 1$. 
    The first points we can identify in $P_n$ are $(0,0,0,0)$ and $(1,2,3,4)$. 
    We can also take the four points $(0,0,0,n)$ for $n \in [4]$. 
    We have $(0,0,1,2)$, $(0,0,1,3)$, and $(0,0,1,4)$; $(0,0,2,3)$ and $(0,0,2,4)$; as well as $(0,0,3,4)$. 
    Finally, we can take $(0,1,2,3)$, $(0,1,2,4)$, $(0,1,3,4)$, and $(0,2,3,4)$. 
    
    This consists of a total of $16$ lattice points and these are the only points in the simplex in the first dilate. 
    Note, not all of these points are vertices, for example, $(0,0,0,2)$ will be on the line between the origin and $(0,0,0,4)$.
    This gives $L_\Z(P_4; 1) = 16 = 2^4 = (1+1)^4$.
    In fact, we will later find the distribution of points in the lecture-hall simplex and our $\mathcal{S}^{\boldsymbol{\tau}_n}$ simplex to be similar. \hfill $\diamondsuit$
\end{example}

In general, the vertex set for $P_n$ consists of exactly the $n+1$ affinely independent points:
$(0,0,0,\ldots,0,0), (0,0,0,\ldots,0,n), (0,0,0,\ldots,n-1,n), \ldots, (0,0,3,\ldots, n-1, n), (0,2,3\ldots,n-1,n),$ and $(1,2,3,\ldots,n-1,n).$
More information about $P_n$, including the proof of the following proposition can be found in \cite{BrightSavage}.

\begin{proposition}[Theorems 1 and 2.2,\cite{BrightSavage}]\label{prop:LHS}
    Let $P_n$ be the lecture-hall simplex of dimension $n$. \[L_\Z(P_n;t) = (t+1)^n.\]
\end{proposition}

Observe that the lattice-point enumerator is equivalent for both the lecture-hall simplex $P_n$ and the $n$-dimensional unit cube  $\square^n$ from Example~\ref{ex:unitcube}. 
We will show that our $\mathcal{S}^{\boldsymbol{\tau}_n}$ simplices from Section \ref{sec:geometry} are also in bijection with these polytopes, and hence have the same lattice-point-enumerator.

\subsection{Notation}\label{subsec:notation}
\text{} 

 We conclude our preliminaries with a notation table for reference. 
Note that it includes notation for terms that will be defined in later sections.

\begin{table}[ht]
    \centering
    \caption{Notation Table}
    \begin{tabular}{|c|c|}
      \hline
        \textbf{Symbol} & \textbf{Meaning} \\ \hline
        $s^i(\boldsymbol{\pi})$ & the $i^{th}$ output from stack-sorting a permutation $\boldsymbol{\pi}$ \\ \hline
        $P^{\circ}$ & the interior of the polytope $P$ \\ \hline
        $L_\Z(P;t)$ & the number of lattice points in $tP$ where $t \in \Z$ \\ \hline
        $\Ehr_\Z(P;z)$ & the ordinary generating function for $L_\Z(P;t)$ \\ \hline
        $h^*_\Z(P;z)$ & the polynomial numerator in the rational form of $\Ehr_\Z(P;z)$ \\ \hline
        $\square^n$ & the $n$-dimensional unit cube \\ \hline
        $P_n$ & the $n$-dimensional lecture-hall simplex \\ \hline
        $\mathcal{L}^n$ & the set of permutations on $[n]$ ending in $n1$ \\ \hline
        $\boldsymbol{\tau}_n$ & the permutation $ 23\cdots n1$ \\ \hline
        $\mathcal{S}^{\boldsymbol{\pi}}$ & the set $\{ s^0(\boldsymbol{\pi}), s^1(\boldsymbol{\pi}), s^2(\boldsymbol{\pi}), \ldots \}$ \\ \hline 
        $\triangle_n$ & the $(n-1)$-dimensional  $\mathcal{S}^{\boldsymbol{\tau}_n}$ simplex  \\ \hline
        $\mathbf{e}$ & the permutation $12\cdots n $  \\ \hline
        $P'$ & the projection of $P \subset \R^n$ onto the hyperplane defined by $x_n = 0$ \\ \hline
       \rule{0pt}{3ex} $\overline{P}$ & the projection of $P \subset \R^n$ into $\R^{n-1}$ by removing each point's last coordinate \\ \hline
        $\rule{0pt}{3ex}\hat{P}$ & the lift of $P \subset \R^n$ into $\R^{n+1}$ by appending a last coordinate of $0$ to each point \\
      \hline
    \end{tabular}
    \label{tab:notation}
\end{table}


\section{Stack-Sorting on \texorpdfstring{$Ln1$}- permutations} \label{sec:Ln1 perms}
In this section we explore the behavior of the stack-sorting algorithm on a special permutation form, one we will find to act in an interesting way with respect to the algorithm.
While the theory of stack-sorting algorithms is rich and comprehensive, this makes parsing through the vast literature to find known results a challenge. 
For self-containment, we include our proofs of the statements that follow, with appropriate references and credit when possible. 
  
\begin{definition} \label{def:Ln1}
Let $\mathcal{L}^n$ denote the set of all permutations of the form $Ln1$, where $L$ is a permutation of $\lbrace 2, 3, \ldots, n-1 \rbrace$.
That is, each element of $\mathcal{L}^n$ is a permutation that ends with $n$ and then $1$. 
\end{definition}
  
\noindent Notice that there are $(n-2)!$ distinct permutations in $\mathcal{L}^n$ as they can be counted by the number of different possibilities for $L$. 
We take special note of the permutation $\tau_n:=23\cdots n1\in \mathcal{L}^n$ in forthcoming results. 
  
\begin{example}
For $n=4$, $2341$ and $3241$ are the two permutations that comprise the set $\mathcal{L}^n$.
\end{example}
  
  \begin{lemma}\label{lem:bonasurvey1.2}
      Let $\boldsymbol{\pi} = \pi_1\pi_2 \cdots \pi_n$ be a permutation of $n$ distinct ordered elements and let $x= \max\lbrace \pi_1,\pi_2, \dots, \pi_n\rbrace$. 
      If $L$ and $R$ are the sub-permutations of $\boldsymbol{\pi}$ such that $\boldsymbol{\pi} = LxR$, then the stack-sorting algorithm yields $s(\boldsymbol{\pi}) = s(L)s(R)x$. 
  \end{lemma}
  
  This result is a generalization of \cite[Lemma 1.2]{bona03}. 
  In other words, the algorithm sorts everything to the left of $x$ (the maximum element) first, regardless of the contents to its right.
  Then, everything to the right of $x$ is sorted and $x$ takes the $n^{\text{th}}$ and final place in the output permutation.

  \begin{proof}
    Consider an arbitrary $\pi_i \in L$. Since $x \not\in L$ and thus $\pi_i \neq x$, we must have $\pi_i < x$.
    This implies that when all of $L$ has been pushed on the stack and $x$ is taken as input by the algorithm, all of $L$ will be popped \emph{before} $x$ is pushed on the stack. 
    After all of $L$ has been popped, $x$ will be pushed onto an empty stack.
    
    Similarly, it can be determined that if $\pi_j \in R$, then $\pi_j < x$.
    Because there is no element greater than $x$ that can pop it from the stack, all of $R$ will be pushed and then subsequently popped in some order before $x$ (at the bottom) is finally popped. 
    Hence, the algorithm will sort $R$ entirely into $s(R)$ before popping $x$; by combining these two results, we obtain that $s(\boldsymbol{\pi}) = s(L)s(R)x$.
  \end{proof}

As a consequence of this lemma, the next corollary follows. 
  
  \begin{corollary} \label{corr:firstiter}
      If $\boldsymbol{\pi} \in \mathcal{L}^n$, then $s(\boldsymbol{\pi}) = s(L)1n$.
Furthermore, $s(\boldsymbol{\pi})$ exactly ends with $n-1$, $1$, then $n$.
\end{corollary}
\noindent Note that the previous corollary can be readily extended to a biconditional.
  
\begin{theorem} \label{thm:iterform}
Let $\boldsymbol{\pi} \in \mathcal{L}^n$. 
For all $i \in [n-2]$, $s^i(\boldsymbol{\pi})$ exactly ends with \[(n-i)1(n-i+1)\ldots(n-2)(n-1)n.\]
\end{theorem} 
  
\begin{proof}
The base case when $i=1$ is the statement of Corollary~\ref{corr:firstiter}.
Assume the theorem holds for $s^{i-1}(\boldsymbol{\pi})$, i.e., it begins with some permutation $\boldsymbol{\rho}$ of $\lbrace 2,3,\ldots,n-i \rbrace$ and ends exactly with $(n-i+1)1(n-i+2)\ldots(n-1)n$. 
      
Consider another application of the stack-sorting algorithm to $s^{i-1}(\boldsymbol{\pi})$. 
Notice that $(n-i+1)$ is greater than all elements in $\boldsymbol{\rho}$.
Thus, when $(n-i+1)$ is considered by the algorithm, all elements in $\boldsymbol{\rho}$ will be popped before $(n-i+1)$ is pushed.
By Lemma~\ref{lem:bonasurvey1.2}, we know the final element in $s(\boldsymbol{\rho})$ will be $n-i$. 
Now, the algorithm sorts $(n-i+1)1(n-i+2)\ldots(n-1)n$ into $1(n-i+1)(n-i+2)\ldots(n-1)n$.

Therefore, we obtain that $s^i(\boldsymbol{\pi})$ exactly ends with $(n-i)1(n-i+1)\ldots(n-1)n$ as desired. 
Further, note that when $i=n-1$ we finally obtain the identity permutation after the maximal number of iterations.
\end{proof}

\begin{example}\label{ex:iterform}
For $n=6$, Theorem \ref{thm:iterform} tells us that $\boldsymbol{\pi} \in \mathcal{L}^6$ ends with $61$, as expected; $s(\boldsymbol{\pi})$ ends with $516$, as stated in Corollary~\ref{corr:firstiter}. 
Further, $s^2(\boldsymbol{\pi})$ ends with $4156$, $s^3(\boldsymbol{\pi})$ ends with $31456$, and thus $s^3(\boldsymbol{\pi}) = 231456$. 
We conclude $s^4(\boldsymbol{\pi}) = 213456$ and $s^5(\boldsymbol{\pi}) = 123456$.
\end{example}

\begin{remark}\label{rem:iterform}
For all $n\geq 4$, if $\boldsymbol{\pi} \in \mathcal{L}^n$, then $s^{n-3}(\boldsymbol{\pi})=2314\cdots n$, $s^{n-2}(\boldsymbol{\pi})=213\cdots n$, and of course $s^{n-1}(\boldsymbol{\pi}) = 123\cdots n$. 
\end{remark}
  
This tells us that regardless of what $\boldsymbol{\pi} \in \mathcal{L}^n$ is, the final three iterations of the algorithm are all the same. 
The permutations ``meet" at this point. 
However, before those last few iterations, we cannot explicitly describe what $s^i(\boldsymbol{\pi})$ will start with without specifying $L$. 
Specifying $L=23\cdots(n-1)$ yields the following useful result.
  
\begin{corollary} \label{corr:iterform}
For $\boldsymbol{\tau}_n\in \mathcal{L}^n$ and $i \in [n-1]$, 
\[s^i(\boldsymbol{\tau}) = 23\cdots(n-i)1(n-i+1)\cdots(n-1)n.\]
\end{corollary}
  
Inputting any sequence in ascending order will have the stack-sorting algorithm output the original sequence.
Thus, the initial segment of the permutation will always be $23\cdots(n-i-1)$.
We not only learn about the maximality of $Ln1$, but can also enumerate the general form of each algorithm iteration in the special case of $\boldsymbol{\tau}_n\in \mathcal{L}^n$. 
  
\begin{definition}
For a permutation $\boldsymbol{\pi} \in \mathfrak{S}_n$, we define
\[\mathcal{S}^{\boldsymbol{\pi}}:=\{ s^0(\boldsymbol{\pi}), s^1(\boldsymbol{\pi}), s^2(\boldsymbol{\pi}), \ldots \}, \]
to be the set of output permutations from each iteration of the stack-sorting algorithm until it terminates at outputting the identity permutation.
\end{definition} 


\begin{example}
Consider $\boldsymbol{\tau_5}=23451\in \mathcal{L}^5$. 
By Corollary~\ref{corr:iterform} we obtain that $s(\boldsymbol{\tau}) = 23415$, $s^2(\boldsymbol{\tau}) = 23145$, $s^3(\boldsymbol{\tau}) = 21345$, and $s^4(\boldsymbol{\tau}) = 12345$.
Notice that $\boldsymbol{\tau}$ is maximal and \[\mathcal{S}^{\boldsymbol{\tau}} = \lbrace 23451, 23415, 23145, 21345, 12345\rbrace.\]
\end{example}

The following theorem appears as Theorem 4.2.4 in Julian West's PhD \cite{West}, which tells us that maximal permutations and $Ln1$ permutations are equivalent. 
We provide an alternate proof here. 
  
\begin{theorem} \label{thm:ssortable}
A permutation $\boldsymbol{\pi} \in \mathcal{L}^n$ if and only if $\boldsymbol{\pi}$ is exactly $(n - 1)$-stack sortable. 
\end{theorem}

\begin{proof}
If $\boldsymbol{\pi} \in \mathcal{L}^n$, the desired consequence is a result of Theorem~\ref{thm:iterform}, which dictates that $s^{n-2}(\boldsymbol{\pi})=213\cdots n$.
It follows $s^{n-1}(\boldsymbol{\pi}) = 123 \cdots n$, therefore $\boldsymbol{\pi}$ is exactly $(n-1)$-stack-sortable.
  
      If $\boldsymbol{\pi} \not\in \mathcal{L}^n$, then $1$ is not in the last place of $\boldsymbol{\pi}$ \emph{or} some $k<n$ is directly left of $1$ in $\boldsymbol{\pi}$. 
      We consider these two cases. 
      
      \begin{case}  
          If $1$ is not the last digit in $\boldsymbol{\pi}$, then $1$ will reach the first position of the output in less than $n-1$ iterations of the algorithm. 
          This is because the $1$ will always move at least one place to the left if it is not already leftmost; the top of the stack will always be greater than $1$, so we will always push it onto the stack where it will be immediately popped before the previous top element.
          The remaining $n-1$ elements will also have been sorted in less than $n-1$ iterations because the maximal possible number of iterations it takes to sort the remaining sequence of $n-1$ elements is $n-2$.
          Thus, $\boldsymbol{\pi}$ can be stack-sorted in less than $n-1$ iterations and is not exactly $(n-1)$-stack-sortable.
      \end{case}
      
      \begin{case}
          If $\boldsymbol{\pi} = Lk1$ such that $k<n$ and $L$ is a permutation of $\lbrace 1, 2, \ldots, k-1, k+1, \ldots, n \rbrace$, then $1$ will move more than one space left after the first iteration of the algorithm.
          This is because $n$ will remain on the bottom of the stack (possibly more elements) while $k$ and $1$ are both pushed and then cleared first, therefore $s(\boldsymbol{\pi})$ will contain $1$ to the left of both $k$ and $n$.
          Similar to the reasoning in the prior case, $1$ will reach the first position of the output in less than $n-2$ more iterations, while the remaining elements will also be sorted in ascending order in less than $n-2$ more iterations by maximality. 
          As before, $\boldsymbol{\pi}$ can be stack-sorted in less than $n-1$ iterations and is not exactly $(n-1)$-stack-sortable.\end{case}\end{proof}


\section{Geometry of \texorpdfstring{$\conv\left(\mathcal{S}^{\boldsymbol{\pi}}\right)$}-} \label{sec:geometry}
In this section we interpret our permutations as points in $\R^n$, which leads us to ask:
What is the geometry that arises by taking the convex hull of stack-sorting iterations for $\boldsymbol{\pi} \in \mathcal{L}^n$?

\begin{proposition}\label{theorem:tai}
For $\boldsymbol{\pi} \in \mathcal{L}^n$, the set $\mathcal{S}^{\boldsymbol{\pi}}$ is affinely independent in $\R^n$.
  \end{proposition}
  
  \begin{proof}
Consider $\boldsymbol{\pi} \in \mathcal{L}^n$.
Let $\be = 123\cdots n = s^{n-1}(\boldsymbol{\pi})$. 
Subtracting $\be$ component-wise from each of the other family members and utilizing Theorem~\ref{thm:iterform}, we deduce that $s^i(\boldsymbol{\pi}) - \be$ must end with $1(1-n+i)0\ldots 000$, where the $1$ is in the $(n-i-1)^{\text{th}}$ position. 
  
Note that 
\[s^{n-2}(\boldsymbol{\pi}) - \be = 1(-1)00\cdots 00 \text{ and } s^{n-3}(\boldsymbol{\pi}) - \be = 11(-2)0\cdots 00\] by Remark~\ref{rem:iterform}. 
As vectors in $\R^n$, these two are linearly independent.
  
Assume $S = \lbrace s^{n-2}(\boldsymbol{\pi}) - \be, s^{n-3}(\boldsymbol{\pi}) - \be, \ldots, s^{n-k+1}(\boldsymbol{\pi}) - \be \rbrace$ is a linearly independent set in $\R^n$ and consider $S \cup \lbrace s^{n-k}(\boldsymbol{\pi}) - \be \rbrace$. 
We know that $s^{n-k}(\boldsymbol{\pi}) - \be$ is the only vector in this set with a nonzero number in the $k^{\text{th}}$ coordinate.
Therefore, no linear combination of the vectors in $S$ yields $s^{n-k}(\boldsymbol{\pi}) - \be$ and $S \cup \lbrace s^{n-k}(\boldsymbol{\pi}) - \be \rbrace$ is linearly independent.
  
We can conclude $\mathcal{S}^{\boldsymbol{\pi}}$  is affinely independent in $\R^n$.
\end{proof}

As a corollary we obtain the following proposition. 
  
\begin{proposition}\label{corr:simpl}
 For $\boldsymbol{\pi} \in \mathcal{L}^n$, the convex hull of $\mathcal{S}^{\boldsymbol{\pi}}$  forms an $(n-1)$-simplex in $\R^n$. 
\end{proposition}

When considering $\mathcal{S}^{\boldsymbol{\tau}_n}$, for $\boldsymbol{\tau}_n:=23\cdots n1$, we will denote its convex hull as $\triangle_n$ and refer to it as the \defterm{$\mathcal{S}^{\boldsymbol{\tau}}$ simplex}.
Notice that the simplex formed is one dimension less than the dimension of its ambient space. 
This is not a trivial detail, it will necessitate some caution as we proceed to develop the Ehrhart theory of $\triangle_n$, especially when comparing it to the unit cube and lecture-hall simplex which are both full-dimensional.

\begin{theorem} \label{thm:hollow}
For $\boldsymbol{\tau}_n:=23\cdots n1$, the simplex $\triangle_n:= \conv(\mathcal{S}^{\boldsymbol{\tau}})$ is hollow. 
Specifically, all non-vertex integer points of $\triangle_n$ lie on the facet formed from the convex hull of $\mathcal{S}^{\boldsymbol{\tau}_n} \setminus \{\mathbf{e}\}$. 
\end{theorem}
  
\begin{proof}
Note that all elements of $\mathcal{F}_{L'n1}$, treated as points in $\R^n$, have $2$ as their fist coordinate entry except for $\mathbf{e} = (1,2, \ldots, n)$. 
This implies that if we take a convex combination of elements of $\mathcal{S}^{\boldsymbol{\tau}}$, the first coordinate entry is of the form 
\[\lambda_1 + 2(\lambda_2 + \dots + \lambda_n) = k,\] where $k$ is a real number and $0< k \leq 2 $ because of the constraints on $\lambda$ from Equation (\ref{eqn: convex hull}).
Further note that $\lambda_2 + \dots + \lambda_n = 1 - \lambda_1$.
Hence, $2 - \lambda_1 = k$, which we can rewrite as $\lambda_1 = 2-k$. 
For our first-coordinate entry to be an integer, we can only consider $k=1$ and $2$.
If $k=1$, then $\lambda_1 = 1$ and we obtain the vertex point $\mathbf{e}$. 
If $k=2$, then $\lambda_1 = 0$ and the integer point exists on the facet containing all vertices except $\mathbf{e}$.
\end{proof}

Hollowness is not a property special to $\triangle_n$. 
The unit cube and lecture-hall simplex also satisfy this property.
Because we know their Ehrhart polynomials, we can argue this much more efficiently using reciprocity:
\[(-1)^d L_\Z(P_n^{\circ}; 1) = L_\Z(P_n;-1)=(-1 + 1)^n = 0. \]
\begin{remark}
The lecture-hall simplex is even more so like $\triangle_n$ because it is hollow in the exact same way.
All non-vertex points are on the facet formed from the convex hull of its vertices excluding $(1,2,\ldots,n)$.
An almost identical argument to Theorem~\ref{thm:hollow} can show this to be the case.
\end{remark}

\begin{remark}
    A lattice polytope $P$ is hollow if and only if $-1$ is a root of its Ehrhart polynomial. 
    More generally, the dilate $tP$ is hollow if and only if $-t$ is a root of the Ehrhart polynomial of $P$.
\end{remark}

To conclude this section, we leave the reader with some directions for further investigation of these subpolytopes of the permutahedron arising from this set-up.

\begin{question}\label{question}
For $\boldsymbol{\pi}\in \mathfrak{S}_n$, $\conv(\mathcal{S}^{\boldsymbol{\pi}})$ forms 
a subpolytope of the permutahedron $\Pi_n$.
\begin{enumerate}
    \item Which choice of $\boldsymbol{\pi}\in \mathfrak{S}_n$ maximizes the volume of $\conv(\mathcal{S}^{\boldsymbol{\pi}})$?

    \item  What is the maximum volume of such a polytope $\conv(\mathcal{S}^{\boldsymbol{\pi}})$? 

    \item Which choice of $\boldsymbol{\pi}\in \mathfrak{S}_n$ guarantees that $\conv(\mathcal{S}^{\boldsymbol{\pi}})$ forms simplex? 

    \item What is the maximum volume of a stack-sorting simplex?

    \item What can be said about triangulations of $\Pi_n$ containing a stack-sorting simplex? 
    
    \item Can we use a triangulation of $\Pi_n$ containing a stack-sorting simplex to effectively determine the lattice-point enumerator of $\Pi_n$?
\end{enumerate}
\end{question}

\begin{figure}[h]
        \centering
\begin{tikzpicture}%
	[scale=.7,
	edge/.style={color=blue},
	facet/.style={fill=cyan,fill opacity=0.300000},
	baseline=(222).base]
%
%
\coordinate (321) at (0.86603, -0.50000);
\node[anchor=west] at (321) {\color{black}\footnotesize$321$};

\coordinate (231) at (0.86603, 0.50000);
\node[anchor=west] at (231) {\color{black}\footnotesize$231$};

\coordinate (312) at (0.00000, -1.00000);
\node[anchor=north] at (312) {\color{black}\footnotesize$312$};

\coordinate (132) at (0.00000, 1.00000);
\node[anchor=south] at (132) {\color{black}\footnotesize$132$};

\coordinate (213) at (-0.86603, -0.50000);
\node[anchor=east] at (213) {\color{black}\footnotesize$213$};

\coordinate (123) at (-0.86603, 0.50000);
\node[anchor=east] at (123) {\color{black}\footnotesize$123$};

\coordinate (222) at (0,0);
\node[anchor=north] at (222) {\color{black}\scriptsize};
\fill[facet] (123) -- (132) -- (231) -- (321) -- (312) -- (213) -- cycle {};
\draw[edge] (321) -- (231);
\draw[edge] (321) -- (312);
\draw[edge] (231) -- (132);
\draw[edge] (312) -- (213);
\draw[edge] (132) -- (123);
\draw[edge] (213) -- (123);

\draw[pink, fill=pink] (123) -- (231) -- (213) --cycle {};

\draw[magenta,fill=magenta] (321) circle (2pt);
\draw[magenta,fill=magenta] (231) circle (2pt);
\draw[magenta,fill=magenta] (312) circle (2pt);
\draw[magenta,fill=magenta] (213) circle (2pt);
\draw[magenta,fill=magenta] (123) circle (2pt);
\draw[magenta,fill=magenta] (222) circle (2pt);
\draw[magenta,fill=magenta] (132) circle (2pt);

\end{tikzpicture}
\caption{The simplex $\conv(\mathcal{S}^{321})$ in the $3$-permutahedron $\Pi_3$.}
\end{figure}  

\begin{problem}\label{problem} 
For $\boldsymbol{\pi}\in \mathfrak{S}_n$ with a particular structure or (avoiding) pattern, establish the following for the polytope $\conv(\mathcal{S}^{\boldsymbol{\pi}})$:
\begin{enumerate}
    \item its face structure/combinatorial type,
    \item a recursive, closed, and/or combinatorial formula for its normalized volume, and 
    \item its Ehrhart theory.
\end{enumerate}    
\end{problem}


\section{Ehrhart theory of \texorpdfstring{$\mathcal{S}^{\boldsymbol{\tau}_n}$}- simplices}\label{sec:ehrhart}
We now develop the Ehrhart theory of the $\mathcal{S}^{\boldsymbol{\tau}_n}$ simplices. 
There is a slight caveat that prevents us from constructing unimodular transformations directly between $\triangle_n$ and $\square^{n-1}$ or $P_{n-1}$ because $\triangle_n$, unlike the other objects being full-dimensional in $\R^{n-1}$, is $(n-1)$-dimensional in $\R^n$. 
This means a matrix representation of $\triangle_n$ will be $n$ by $n$, while that of say $P_{n-1}$ will be $n-1$ by $n$ or $\square^{n-1}$ will be $n-1$ by $2^{n-1}$. 
Any transformation between them is a non-square matrix and thus cannot be unimodular. 

We transform from $\triangle_n$ to $P_{n-1}$ instead of $\square^{n-1}$ because the former is also a simplex of the same dimension, which implies they both have $n$ vertices. 
We then need to make a correction for the difference in their ambient spaces.
One option, the one we will proceed with, is to lift $P_{n-1}$ to $\R^n$ by appending a new last coordinate that does not affect the lattice point count. 
We will establish a unimodular transformation from this new construction to $\triangle_n$.
Another option is to project $\triangle_n$ onto a hyperplane, ensuring the lattice count remains preserved, and then transform this construction in $\R^{n-1}$ to $P_{n-1}$.
Regardless, to construct a unimodular transformation we need to compare two polytopes living in the same ambient space.

Before we proceed, a quick observation about the recursive structure of the lecture-hall simplex will be insightful.
Let $\mathbf{0_n}$ be the zero vector in $\R^n$.

\begin{proposition}\label{rem:LHrec}
    Let $V_n$ denote the vertex set of $P_n$ and $Q_n := \conv(V_n \setminus \{\mathbf{0_n}\})$.
   Then
    \[ L_\Z(P_n; t) = L_\Z(Q_{n+1}; t).\]
In other terms, the polytopes $\conv(V_n)$ and $\conv(V_{n+1} \setminus \{\mathbf{0_{n+1}}\})$ are integrally equivalent.
\end{proposition}

\begin{proof}
    Consider $V_n$, which consists of the points $\mathbf{0_n}, (0,0,0,\ldots,0,n), (0,0,0,\ldots,n-1,n), \ldots,$\\ $(0,0,3,\ldots, n-1, n), (0,2,3\ldots,n-1,n), \text{ and }(1,2,3,\ldots,n-1,n).$. 
    Append a last coordinate of $n+1$ to each vertex to obtain a polytope in $\R^{n+1}$ that is integrally equivalent to the original. Notice this new set is $V_{n+1}$ without $\mathbf{0_{n+1}}$.
\end{proof}

\begin{theorem}\label{thm:integral_equivalence}
The $\mathcal{S}^{\boldsymbol{\tau}_n}$ simplex $\triangle_{n}$ and $(n-1)$-dimensional lecture hall simplex $P_{n-1}$ are integrally equivalent. 
In particular, \[L_\Z(\triangle_{n}; t) = L_\Z(P_{n-1}; t).\]
\end{theorem}

We employ the notation $S + \bv$ when appropriate to denote adding the point $\bv$ to each point in the set $S$. 
If $M$ is a matrix, $M + \bv$ adds $\bv$ to each column of $M$.

\begin{proof}
    Let $M$ be the $n$ by $n$ matrix representation of $\triangle_n$ with column $i$ as $s^{n-i}(\boldsymbol{\tau}_n)$. 
    Further, let $L$ be the $n$ by $n$ lower-triangular matrix of $-1$'s, let $v_k = \sum_{i=2}^k i$, and let $\bv_k = (v_2, v_3, \ldots, v_{k+1})$. 
    
    We have that $L\cdot M + \bv_n$ has columns $(1,2,\ldots,n), (0,2,\ldots,n), \ldots (0,0,\ldots,n)$. This is the vertex set $V_n$ of $P_n$ without $\mathbf{0_n}$. Recall from Proposition~\ref{rem:LHrec} that $Q_n = \conv(V_n \setminus \{\mathbf{0_n}\})$ yields a polytope integrally equivalent to $P_{n-1}$. We now have a unimodular transformation from $\triangle_n$ to an integrally equivalent polytope to $P_{n-1}$, from which we obtain:
    \[ L_\Z(\triangle_{n}; t) = L_\Z(Q_n-\bv;t) = L_\Z(Q_n;t) = L_\Z(P_{n-1}; t).\]
\end{proof}

\begin{example}
Observe that in Figure \ref{fig:three_polytopes}, the three polytopes, the $2341$ simplex $\triangle_4\subseteq \R^4$, the 3-unit cube $\square^3$, and the 3-dimensional lecture-hall simplex $P_{3}$, all share the same lattice-point count in their first dilate. 
Also note that $\triangle_4$ is a 3-dimensional polytope in $\R^4$, whereas the other two polytopes are full-dimensional in their ambient space. Notice the similarity between how points are distributed in $P_3$ and $\triangle_3$.
\end{example}

\begin{figure}[!ht]
    \centering
    \scalebox{1.3}{\begin{tikzpicture}[x  = {(0.9cm,-0.076cm)},
                    y  = {(-0.06cm,0.95cm)},
                    z  = {(-0.44cm,-0.29cm)},
                    scale = .7,
                    color = {lightgray}]

  \coordinate (v0_P) at (2, 3, 4);
  \coordinate (v1_P) at (2, 3, 1);
  \coordinate (v2_P) at (2, 1, 3);
  \coordinate (v3_P) at (1, 2, 3);

  \definecolor{edgecolor_P}{rgb}{ 0,0,0 }
  \tikzstyle{facestyle_P} = [fill=andresblue, fill opacity=0.3, preaction={draw=white, line cap=round, line width=1.5 pt}, draw=edgecolor_P, line width=1 pt, line cap=round, line join=round]

  \draw[facestyle_P] (v1_P) -- (v3_P) -- (v0_P) -- (v1_P) -- cycle;
  \draw[facestyle_P] (v2_P) -- (v3_P) -- (v1_P) -- (v2_P) -- cycle;
  \draw[facestyle_P] (v2_P) -- (v1_P) -- (v0_P) -- (v2_P) -- cycle;


  \draw[facestyle_P] (v0_P) -- (v3_P) -- (v2_P) -- (v0_P) -- cycle;



  \definecolor{pointcolor_LatticepointsandverticesofP_0}{rgb}{1,0,1}
  \definecolor{pointcolor_LatticepointsandverticesofP_1}{rgb}{1,0,1}
  \definecolor{pointcolor_LatticepointsandverticesofP_2}{rgb}{1,0,1}
  \definecolor{pointcolor_LatticepointsandverticesofP_3}{rgb}{1,0,1}
  \definecolor{pointcolor_LatticepointsandverticesofP_4}{rgb}{1,0,1}
  \definecolor{pointcolor_LatticepointsandverticesofP_5}{rgb}{1,0,1}
  \definecolor{pointcolor_LatticepointsandverticesofP_6}{rgb}{1,0,1}
  \definecolor{pointcolor_LatticepointsandverticesofP_7}{rgb}{1,0,1}

  \coordinate (v0_LatticepointsandverticesofP) at (1, 2, 3);
  \coordinate (v1_LatticepointsandverticesofP) at (2, 1, 3);
  \coordinate (v2_LatticepointsandverticesofP) at (2, 2, 2);
  \coordinate (v3_LatticepointsandverticesofP) at (2, 2, 3);
  \coordinate (v4_LatticepointsandverticesofP) at (2, 3, 1);
  \coordinate (v5_LatticepointsandverticesofP) at (2, 3, 2);
  \coordinate (v6_LatticepointsandverticesofP) at (2, 3, 3);
  \coordinate (v7_LatticepointsandverticesofP) at (2, 3, 4);

  \fill[pointcolor_LatticepointsandverticesofP_7] (v7_LatticepointsandverticesofP) circle (2 pt);

  \fill[pointcolor_LatticepointsandverticesofP_0] (v0_LatticepointsandverticesofP) circle (2 pt);

  \fill[pointcolor_LatticepointsandverticesofP_1] (v1_LatticepointsandverticesofP) circle (2 pt);

  \fill[pointcolor_LatticepointsandverticesofP_3] (v3_LatticepointsandverticesofP) circle (2 pt);

  \fill[pointcolor_LatticepointsandverticesofP_6] (v6_LatticepointsandverticesofP) circle (2 pt);

  \fill[pointcolor_LatticepointsandverticesofP_2] (v2_LatticepointsandverticesofP) circle (2 pt);

  \fill[pointcolor_LatticepointsandverticesofP_5] (v5_LatticepointsandverticesofP) circle (2 pt);

  \fill[pointcolor_LatticepointsandverticesofP_4] (v4_LatticepointsandverticesofP) circle (2 pt);

\end{tikzpicture}
\qquad

\begin{tikzpicture}[x  = {(0.9cm,-0.076cm)},
                    y  = {(-0.06cm,0.95cm)},
                    z  = {(-0.44cm,-0.29cm)},
                    scale = 1,
                    color = {lightgray}]

  \coordinate (v0_c) at (0, 0, 0);
  \coordinate (v1_c) at (1, 0, 0);
  \coordinate (v2_c) at (0, 1, 0);
  \coordinate (v3_c) at (1, 1, 0);
  \coordinate (v4_c) at (0, 0, 1);
  \coordinate (v5_c) at (1, 0, 1);
  \coordinate (v6_c) at (0, 1, 1);
  \coordinate (v7_c) at (1, 1, 1);

  \definecolor{edgecolor_c}{rgb}{ 0,0,0 }
  \tikzstyle{facestyle_c} = [fill=andresblue, fill opacity=0.3, preaction={draw=white, line cap=round, line width=1.5 pt}, draw=edgecolor_c, line width=1 pt, line cap=round, line join=round]

  \draw[facestyle_c] (v5_c) -- (v4_c) -- (v0_c) -- (v1_c) -- (v5_c) -- cycle;
  \draw[facestyle_c] (v0_c) -- (v4_c) -- (v6_c) -- (v2_c) -- (v0_c) -- cycle;
  \draw[facestyle_c] (v0_c) -- (v2_c) -- (v3_c) -- (v1_c) -- (v0_c) -- cycle;


  \draw[facestyle_c] (v2_c) -- (v6_c) -- (v7_c) -- (v3_c) -- (v2_c) -- cycle;


  \draw[facestyle_c] (v7_c) -- (v5_c) -- (v1_c) -- (v3_c) -- (v7_c) -- cycle;


  \draw[facestyle_c] (v6_c) -- (v4_c) -- (v5_c) -- (v7_c) -- (v6_c) -- cycle;



  \definecolor{pointcolor_Latticepointsandverticesofc_0}{rgb}{1,0,1}
  \definecolor{pointcolor_Latticepointsandverticesofc_1}{rgb}{1,0,1}
  \definecolor{pointcolor_Latticepointsandverticesofc_2}{rgb}{1,0,1}
  \definecolor{pointcolor_Latticepointsandverticesofc_3}{rgb}{1,0,1}
  \definecolor{pointcolor_Latticepointsandverticesofc_4}{rgb}{1,0,1}
  \definecolor{pointcolor_Latticepointsandverticesofc_5}{rgb}{1,0,1}
  \definecolor{pointcolor_Latticepointsandverticesofc_6}{rgb}{1,0,1}
  \definecolor{pointcolor_Latticepointsandverticesofc_7}{rgb}{1,0,1}

  \coordinate (v0_Latticepointsandverticesofc) at (0, 0, 0);
  \coordinate (v1_Latticepointsandverticesofc) at (0, 0, 1);
  \coordinate (v2_Latticepointsandverticesofc) at (0, 1, 0);
  \coordinate (v3_Latticepointsandverticesofc) at (0, 1, 1);
  \coordinate (v4_Latticepointsandverticesofc) at (1, 0, 0);
  \coordinate (v5_Latticepointsandverticesofc) at (1, 0, 1);
  \coordinate (v6_Latticepointsandverticesofc) at (1, 1, 0);
  \coordinate (v7_Latticepointsandverticesofc) at (1, 1, 1);

  \fill[pointcolor_Latticepointsandverticesofc_1] (v1_Latticepointsandverticesofc) circle (2 pt);

  \fill[pointcolor_Latticepointsandverticesofc_3] (v3_Latticepointsandverticesofc) circle (2 pt);

  \fill[pointcolor_Latticepointsandverticesofc_5] (v5_Latticepointsandverticesofc) circle (2 pt);

  \fill[pointcolor_Latticepointsandverticesofc_7] (v7_Latticepointsandverticesofc) circle (2 pt);

  \fill[pointcolor_Latticepointsandverticesofc_0] (v0_Latticepointsandverticesofc) circle (2 pt);

  \fill[pointcolor_Latticepointsandverticesofc_2] (v2_Latticepointsandverticesofc) circle (2 pt);

  \fill[pointcolor_Latticepointsandverticesofc_4] (v4_Latticepointsandverticesofc) circle (2 pt);

  \fill[pointcolor_Latticepointsandverticesofc_6] (v6_Latticepointsandverticesofc) circle (2 pt);

\end{tikzpicture}

\qquad

\begin{tikzpicture}[x  = {(0.9cm,-0.076cm)},
                    y  = {(-0.06cm,0.95cm)},
                    z  = {(-0.44cm,-0.29cm)},
                    scale = .9,
                    color = {lightgray}]

  \coordinate (v0_P) at (0, 0, 0);
  \coordinate (v1_P) at (0, 0, 3);
  \coordinate (v2_P) at (0, 2, 3);
  \coordinate (v3_P) at (1, 2, 3);

  \definecolor{edgecolor_P}{rgb}{ 0,0,0 }
  \tikzstyle{facestyle_P} = [fill=andresblue, fill opacity=0.3, preaction={draw=white, line cap=round, line width=1.5 pt}, draw=edgecolor_P, line width=1 pt, line cap=round, line join=round]

  \draw[facestyle_P] (v0_P) -- (v1_P) -- (v2_P) -- (v0_P) -- cycle;
  \draw[facestyle_P] (v0_P) -- (v2_P) -- (v3_P) -- (v0_P) -- cycle;
  \draw[facestyle_P] (v3_P) -- (v1_P) -- (v0_P) -- (v3_P) -- cycle;


  \draw[facestyle_P] (v2_P) -- (v1_P) -- (v3_P) -- (v2_P) -- cycle;



  \definecolor{pointcolor_LatticepointsandverticesofP_0}{rgb}{1,0,1}
  \definecolor{pointcolor_LatticepointsandverticesofP_1}{rgb}{1,0,1}
  \definecolor{pointcolor_LatticepointsandverticesofP_2}{rgb}{1,0,1}
  \definecolor{pointcolor_LatticepointsandverticesofP_3}{rgb}{1,0,1}
  \definecolor{pointcolor_LatticepointsandverticesofP_4}{rgb}{1,0,1}
  \definecolor{pointcolor_LatticepointsandverticesofP_5}{rgb}{1,0,1}
  \definecolor{pointcolor_LatticepointsandverticesofP_6}{rgb}{1,0,1}
  \definecolor{pointcolor_LatticepointsandverticesofP_7}{rgb}{1,0,1}

  \coordinate (v0_LatticepointsandverticesofP) at (0, 0, 0);
  \coordinate (v1_LatticepointsandverticesofP) at (0, 0, 1);
  \coordinate (v2_LatticepointsandverticesofP) at (0, 0, 2);
  \coordinate (v3_LatticepointsandverticesofP) at (0, 0, 3);
  \coordinate (v4_LatticepointsandverticesofP) at (0, 1, 2);
  \coordinate (v5_LatticepointsandverticesofP) at (0, 1, 3);
  \coordinate (v6_LatticepointsandverticesofP) at (0, 2, 3);
  \coordinate (v7_LatticepointsandverticesofP) at (1, 2, 3);

  \fill[pointcolor_LatticepointsandverticesofP_3] (v3_LatticepointsandverticesofP) circle (2 pt);

  \fill[pointcolor_LatticepointsandverticesofP_5] (v5_LatticepointsandverticesofP) circle (2 pt);

  \fill[pointcolor_LatticepointsandverticesofP_6] (v6_LatticepointsandverticesofP) circle (2 pt);

  \fill[pointcolor_LatticepointsandverticesofP_7] (v7_LatticepointsandverticesofP) circle (2 pt);

  \fill[pointcolor_LatticepointsandverticesofP_2] (v2_LatticepointsandverticesofP) circle (2 pt);

  \fill[pointcolor_LatticepointsandverticesofP_4] (v4_LatticepointsandverticesofP) circle (2 pt);

  \fill[pointcolor_LatticepointsandverticesofP_1] (v1_LatticepointsandverticesofP) circle (2 pt);

  \fill[pointcolor_LatticepointsandverticesofP_0] (v0_LatticepointsandverticesofP) circle (2 pt);

\end{tikzpicture}
}
    \caption{From Left to Right: the $\mathcal{S}^{\boldsymbol{\tau}_4}$-simplex $\triangle_4\subseteq \R^4$, the 3-unit cube $\square^3$, and the 3-dimensional lecture-hall simplex $P_{3}$.}
    \label{fig:three_polytopes}
\end{figure}

Since the Ehrhart polynomial for the unit cube and the lecture-hall simplex is known, we quickly obtain the following result as a corollary. 

\begin{corollary}
    The lattice point enumerator of the $\mathcal{S}^{\boldsymbol{\tau}_n}$ simplex is \[L_\Z(\triangle_n;t) = (t+1)^{n-1}.\]
Furthemore, the Ehrhart series of the $\mathcal{S}^{\boldsymbol{\tau}_n}$  simplex is
    \[\Ehr_\Z(\triangle_n; z) = \frac{\sum_{k=0}^{n-1} A(n,k) z^k}{(1-z)^{n-1}}\]
    where $A(n,k)$ is the number of permutations of $[n]$ with $k$ descents.
\end{corollary}

\begin{remark}
    By Ehrhart-Macdonald reciprocity, the following holds:
    \begin{align*}
        L_\Z(\triangle_n^{\circ};t+2) &= (-1)^{n-1}L_\Z(\triangle_n;-t-2)\\
        &= (-1)^{n-1}(-t-1)^{n-1}
        = (t+1)^{n-1}
        = L_\Z(\triangle_n;t).
    \end{align*}
    Also note that $L_\Z(\triangle_n^{\circ} ; 1) = 0$ and $L_\Z(\triangle_n^\circ ; 2) = 1$.
    Therefore, $\triangle_n$ is Gorenstein of index $2$.
\end{remark}

\section{Real lattice-point enumerators for recursive structures}

Initial efforts to obtain the lattice-point enumerator for $\mathcal{S}^{\boldsymbol{\tau}_n}$ simplices resulted in differing techniques than those used in the previous section. 
While a less direct method, the approach used in the following results yielded some connections to \emph{real} lattice-point enumerators of various translations of our simplices which we present now.

Note that two integrally equivalent polytopes, as defined in Definition \ref{def:integralequiv}, may not have the same \emph{real} lattice-point count. 
For example, $\triangle_n$ and $\triangle_n - \boldsymbol{\tau}_n$ are integrally equivalent, but have different real-lattice point enumerators, as we will see.
This is in notable contrast to the fact integer translations are lattice invariant for lattice polytopes.

Taking the last coordinate in each lattice point of a polytope $P \subset \R^n$ and setting it equal to $0$ is the \emph{projection} of $P$ to the hyperplane defined by $x_n = 0$. 
We denote this operation as $\proj_{x_n=0}(P)$ and use an apostrophe $P'$ as shorthand.

\begin{lemma} \label{lem5.5}
    Let $\bp = (p_1, \ldots , p_n) \in \aff(\triangle_n)$, and $\bp' = (p_1, \ldots , p_{n-1}, 0)$. 
    Consider the $\mathcal{S}^{\boldsymbol{\tau}_n}$ simplex $\triangle_n$ and let $\triangle_n':= \proj_{(x_n = 0)}(\triangle_n)$.
    Then, for any $\lambda \in \R_{\geq 0}$:
   \[
   L_\R(\triangle_n - \mathbf{p} ; \lambda) = L_\R(\triangle_n' - \bp'; \lambda).
   \]
\end{lemma}

We shift $\triangle_n$ by $\bp$ (and thus $\triangle_n'$ by $\bp'$) as a slight extra step in order to show the much stronger result that these two shifted polytopes are in bijection for \emph{all} real dilates $\lambda$ instead of only integer dilates. 
Our desired result for the unshifted polytopes will come as a corollary.

\begin{proof}
    Let $X:= \lambda(\triangle_n - \bp) \cap \Z^n$ and $X':= \lambda(\triangle_n' - \bp') \cap \Z^n$. 
    Define 
    \[\varphi: X \to X' \text{ such that } \varphi(x_1,\ldots,x_{n-1}, x_n) = (x_1,\ldots,x_{n-1},0).\]
    This map is well-defined because of our construction of $X'$ from $X$; it is immediate that if \\ $(x_1,\ldots,x_n)~\in~X$ then $(x_1, \ldots, x_{n-1}, 0) \in X'$ because $X'$ is the projection of each point in $X$ to the hyperplane $x_n=0$.
    
    To prove injectivity, suppose $\bx,\by \in X$ such that $\varphi(\bx) = \varphi(\by)$.
    Thus, \[(x_1, \ldots, x_{n-1}, 0) = (y_1, \ldots, y_{n-1}, 0) \text{ and } x_i = y_i \text{ for } i \in [n-1].\]
    Since the vertices of $\triangle_n$ are permutations of $[n]$, points $\mathbf{z}=(z_1,\dots, z_n)$ in $\triangle_n$ exist on the hyperplane defined by \[\sum_{k=1}^n z_k = \sum_{k=1}^n k = \frac{n(n+1)}{2}.\] 
    Therefore, any point $\mathbf{z}$ in $\triangle_n - \bp'$ lives in $\sum_{k=1}^n z_k = 0$, as does any nonnegative real dilate $\sum_{k=1}^n \lambda z_k = \lambda\sum_{k=1}^n z_k = 0$. 
    Hence, $ \sum_{k=1}^n x_k = \sum_{k=1}^n y_k$ and further $x_n = y_n$. We conclude that $\bx=\by$ and $\varphi$ is injective.

    For surjectivity, consider $\bx' = (x_1,\ldots, x_{n-1}, 0) \in X'$. We know there exists a unique convex combination of vertices of $\lambda(\triangle_n' - \bp')$ such that $\sum_{k=1}^n c_k \bv_k' = \bx'$, where each $c_i \in \R_{\geq 0}$ and $\sum_{k=1}^n c_k = 1$.
    An integral point in $\lambda(\triangle_n - \bp)$ is a convex combination of the vertices of $\lambda(\triangle_n - \bp)$ and similarly for $\lambda(\triangle'_n - \bp')$.
    Note that the corresponding vertices of both polytopes share the exact same first $n-1$ coordinates; thus, the convex combination will yield $x_1, \ldots, x_{n-1}$ as the first $n-1$ coordinates of the resulting point $\bx \in \lambda(\triangle_n - \bp)$. 
    Further, this point lies in the affine span of $\lambda(\triangle_n - \bp)$ and thus we know $\sum_{i=1}^n x_i = 0$. 
    This implies $x_n = -\sum_{i=1}^{n-1} x_i$ and therefore $x_n \in \Z$ because $x_1,\ldots,x_{n-1} \in \Z$. We conclude $\bx \in X$ and $\varphi$ is surjective.
    
    Hence, we have established a bijection between $X$ and $X'$ for all nonnegative real dilates $\lambda$ and can now deduce the polytopes $\lambda(\triangle_n - \bp)$ and $\lambda(\triangle_n' - \bp')$ will always have the same number of lattice points.
\end{proof}

\begin{lemma} \label{lem:realtranslation}
  Let $\triangle_{n}$ be the $\mathcal{S}^{\boldsymbol{\tau}_n}$ simplex and $\overline{\boldsymbol{\tau}}_n:= 23\cdots n$. 
  For all $t \in \N$,
  \[L_{\R}\left(\triangle_n-\overline{\boldsymbol{\tau}_{n+1}}; \frac{t}{n}\right) = L_{\R}\left(\triangle_n-\boldsymbol{\tau}_n; \frac{t}{n}\right).\]
  \begin{proof}
    Recall integer translations preserve lattice count and note that $\boldsymbol{\tau}_n = \overline{\boldsymbol{\tau}_{n+1}} - (0,\ldots,0,n)$. 
    Thus, translating $\triangle_n - \boldsymbol{\tau}_n$ by $(0,\ldots,0,n)$ yields $\triangle_n - \overline{\boldsymbol{\tau}_{n+1}}$. 
    We must conclude they have equivalent lattice counts for all integer dilates $t$. 
    Further, translation by $(0,\ldots,0,\lambda n)$ will remain integral after dilation for all rational dilates with form $\lambda = \frac{t}{n}$, making it lattice-preserving in these special cases as well.
  \end{proof}
\end{lemma}

The following result presents a recursive relationship that allows for the enumeration of lattice points of \emph{real} dilates of $\triangle_{n+1}$ from \emph{rational} dilates of $\triangle_n$.

\begin{theorem} \label{thm:sssrecurrence}
For all $\lambda \in \R_{\geq 0}$,
  \begin{align*}
      L_\R(\triangle_{n+1}- \boldsymbol{\tau}_{n+1};\lambda) &= 
  \left|\lambda((\triangle_{n+1} - \boldsymbol{\tau}_{n+1}) \cap \Z^{n+1})\right| \\ &= \sum_{k=0}^{\lfloor n\lambda\rfloor}\ \left|\frac{k}{n}(\triangle_n-\boldsymbol{\tau}_n) \cap \Z^n\right|=\sum_{k=0}^{\lfloor n\lambda\rfloor}\ L_\R \left(\triangle_n-\boldsymbol{\tau}_n;\frac{k}{n}\right).
  \end{align*}
\end{theorem}

The proof of this theorem is presented in Subsection \ref{subsec:pfthm5.8}. 
An example of the theorem is given to provide some insight to the readers.

\begin{figure}[!ht]
\centering
\scalebox{0.8}{

\begin{tikzpicture}[scale=0.75]
\definecolor{polytopecolor}{rgb}{0.0, 0.72, 0.92}

\draw[ultra thin, fill=gray!30] (2,4) -- (2,-6) -- (-3,-1) -- (2,4);

\draw[ultra thin, fill=polytopecolor!40] (2,4) -- (2,0) -- (0,2) -- (2,4)  node[label={[shift={(-2.13,-1.45)}, text=blue] \Large $\triangle_3 - \boldsymbol{\tau}_3$}] {};


\draw[line width=3pt, orange] (1,3) -- (2,2) node[label={[shift={(1.5,-.2)}, text=orange] \Large ``$\frac{1}{2}\left(\triangle_2-\boldsymbol{\tau}_2\right)$"}] {};

\draw[line width=3pt, orange] (0,2) -- (2,0) node[label={[shift={(1.5,-.2)}, text=orange] \Large ``$\frac{2}{2}\left(\triangle_2-\boldsymbol{\tau}_2\right)$"}] {};

\draw[line width=3pt, orange] (-1,1) -- (2,-2) node[label={[shift={(1.5,-.2)}, text=orange] \Large ``$\frac{3}{2}\left(\triangle_2-\boldsymbol{\tau}_2\right)$"}] {};

\draw[line width=3pt, orange] (-2,0) -- (2,-4)  node[label={[shift={(1.5,-.2)}, text=orange] \Large ``$\frac{4}{2}\left(\triangle_2-\boldsymbol{\tau}_2\right)$"}] {};

\draw[line width=3pt, orange] (-3,-1) -- (2,-6) node[label={[shift={(1.5,-.2)}, text=orange] \Large ``$\frac{5}{2}\left(\triangle_2-\boldsymbol{\tau}_2\right)$"}] {};

\draw[line width=2pt, gray] (-3,-1) -- (2,4) -- (2,-6);


\draw (2,4)
node[fill, color=andrespink,circle, minimum size,label={[shift={(0.65,-.25)}]\small (0,0,0)}] {};

\draw (2,2)
node[fill, color=andrespink,circle, minimum size] {};

\draw (2,0)
node[fill, color=andrespink,circle, minimum size] {};

\draw (2,-2)
node[fill, color=andrespink,circle, minimum size] {};

\draw (2,-4)
node[fill, color=andrespink,circle, minimum size] {};

\draw (2,-6)
node[fill, color=andrespink,circle, minimum size, label={[shift={(0.65,-.75)}]\small (0,-5,5)}] {};

\draw (0,2)
node[fill, color=andrespink,circle, minimum size] {};

\draw (0,0)
node[fill, color=andrespink,circle, minimum size] {};

\draw (0,-2)
node[fill, color=andrespink,circle, minimum size, label={[shift={(-2.3,1.6)}]\Large $\frac{5}{2}\left(\triangle_3-\boldsymbol{\tau}_3 \right)$}] {};

\draw (0,-4)
node[fill, color=andrespink,circle, minimum size] {};

\draw (-2,0)
node[fill, color=andrespink,circle, minimum size] {};

\draw (-2,-2)
node[fill, color=andrespink,circle, minimum size] {};

\end{tikzpicture}

}
\caption{A lattice translate of the $\mathcal{S}^{\boldsymbol{\tau}_3}$ simplex, $\triangle_3-\boldsymbol{\tau}_3\subset \R^3$, is in {\color{blue}{blue}}, lattice points are in {\color{magenta}{magenta}}, $\frac{5}{2}(\triangle_3-\boldsymbol{\tau}_3)$ is in {\color{gray}{gray}}, and the (lifted) five rational dilates of $\triangle_2-\boldsymbol{\tau}_2$ are in {\color{orange}{orange}}.}
\label{fig:recurrence}
\end{figure}

\begin{example}\label{ex:recurrence}
Take $n=2$, then $\boldsymbol{\tau}_3=231$ and $\boldsymbol{\tau}_2=21$ and consider $\lambda=\frac{5}{2}$. 
Note that \[\left|\frac{5}{2}(\triangle_{3} - \boldsymbol{\tau}_3) \cap \Z^{3}\right|=
12,\] which we can directly count from the schematic in Figure \ref{fig:recurrence}. 
Observe that $\triangle_3-\boldsymbol{\tau}_3$ lives in $\R^3$, but is a 2-dimensional simplex. 
By Theorem \ref{thm:sssrecurrence}, we can alternatively count the number of lattice points of $\frac{5}{2}(\triangle_{3} - \boldsymbol{\tau}_3)$ by counting the number of lattice points in the five dilates of $\frac{k}{2}(\triangle_2-\boldsymbol{\tau}_2)$ for $k=1,2,3,4,$ and $5$.
In the schematic, consider the polytopes on their affine span and we can treat the polytopes as being projected onto the hyperplane $x_3=0$ (hence, the quotations in some of the figure labels).
Counting the number of lattice points in the five dilates of $\frac{k}{2}(\triangle_2-\boldsymbol{\tau}_2)$ for $k=1,2,3,4,$ and $5$ as shown in orange, we also obtain 12 lattice points.  
\end{example}

We continue by presenting a recursive relationship that deals with the enumeration of relative interior lattice points
of real dilates of $\triangle_{n+1}$ from \emph{rational} dilates of $\triangle_n$.

\begin{proposition} \label{prop:interiorrecurrence}
  For all $\lambda \in \R_{\geq 0}$,
  \begin{align*}
      L_\R((\triangle_{n+1} - \boldsymbol{\tau}_{n+1})^\circ ; \lambda) &= \left|(\lambda(\triangle_{n+1} - \boldsymbol{\tau}_{n+1}))^{\circ} \cap \Z^{n+1}\right| \\ &= \sum_{k=1}^{\lceil n \lambda \rceil - 1}\left|\left(\frac{k}{n}(\triangle_n-\boldsymbol{\tau}_n)\right)^{\circ} \cap \Z^n\right| =\sum_{k=1}^{\lceil n \lambda \rceil - 1} L_\R\left((\triangle_n -\boldsymbol{\tau}_n)^\circ ; \frac{k}{n} \right).
  \end{align*}
\end{proposition}

\begin{proof}
The proof follows similar arguments to that of the proof of Theorem \ref{thm:sssrecurrence}; modifying that proof, the claim holds.
\end{proof}

We conclude this section by making use of the similarities in the recurrence relations of Theorem \ref{thm:sssrecurrence} and Proposition \ref{prop:interiorrecurrence} to prove the following result by induction.
The result shows that the real lattice-point count for the $\lambda$-th dilate of $\triangle_n - \boldsymbol{\tau}_n$ coincides with the relative interior lattice-point count for the $(\lambda + 2)$-th dilate of $\triangle_n -\boldsymbol{\tau}_n$ for every real number $\lambda$. 
Furthermore, the result shows that $\triangle_n$ is Gorenstein of index $2$. This alternative proof is included because it allows for the Gorenstein result without knowledge of the entire lattice-point enumerator. 

\begin{theorem}\label{thm:gorenstein}
  For $\lambda \in \R$,
  \[
    L_\R(\triangle_n - \boldsymbol{\tau}_n ; \lambda) = L_\R((\triangle_n - \boldsymbol{\tau}_n)^{\circ}; \lambda+2).
  \]
Hence, any integer translate of $\triangle_n - \boldsymbol{\tau}_n$, in particular $\triangle_n$, is Gorenstein of index $2$.
\end{theorem}

The proof of this theorem is presented in Subsection \ref{subsec:pfthm5.11}.

\begin{remark}
If we extend the previous definition and define a polytope to be \defterm{real Gorenstein of index $k$} for an integer $k$, if the relations from Equation \ref{eq:gorenstein}  hold for all real dilates $t > k$, we believe this matches the rational/real Gorenstein definitions given in \cite{BeckEliaRehberg} and we can say that $\triangle_n - \boldsymbol{\tau}_n$ is ``real Gorenstein'' of index $2$.     
\end{remark}

\subsection{Proof of Theorem \ref{thm:sssrecurrence}}\label{subsec:pfthm5.8}

Before beginning, it is necessary to outline a few more terms:

Given a collection of points $C=\lbrace\bx_1,\dots,\bx_k\rbrace$ in $\R^n$, we define a \emph{pointed cone} with apex $\bv\in \R^n$ as a set of the form \begin{equation}
    \cone_\bv({C}):=\left\lbrace \bv + \sum_{i=1}^k \lambda_i (\bx_i-\bv): \lambda_i\in \R_{\geq 0} \right\rbrace.
\end{equation}
This creates a cone that starts from the point $\bv$ and consists of the infinite region bounded by the $k$ rays formed from $\bv$ through a given point $\bx_i$.


Note that the choice of point $\bv$ will affect the information that is encoded in the cone. For a polytope $P$ that is not full-dimensional, it is a careful selection of $\bv \not\in \aff(P)$ that allows dilates of $P$ to perfectly encode themselves in `slices' of the cone over $\bv$.

\begin{definition}\label{S_halfspace}
Let $S \subseteq \R^n$ be an $(n-1)$-dimensional affine subspace and take $\bx \in \R^n - S$.
Then the \defterm{$S$-halfspace with $\bx$} denoted as $\mathcal{H}_{\bx}(S)$ is defined as the unique halfspace in $\R^n$ formed from the hyperplane $S$ that includes the point $\bx$.    
\end{definition}


For the proof of Theorem \ref{thm:sssrecurrence}, we express the real dilate of some translation of the $n$-dimensional $\mathcal{S}^{\boldsymbol{\tau}_n}$-simplex, $\lambda (\triangle_{n+1} - \boldsymbol{\tau}_{n+1})$, as a disjoint union of continuous parallel slices and ``floss out'' the slices without a lattice point. 
We can think of the proof process as taking a MRI scan of $\lambda (\triangle_{n+1} - \boldsymbol{\tau}_{n+1})$, and only observing images of the cross-sections that has a lattice point, in other words, non-empty. 
It turns out that the lattice points in each non-empty slices of $\lambda (\triangle_{n+1} - \boldsymbol{\tau}_{n+1})$ are in bijective correspondence with the lattice points in some rational dilate of $\triangle_n - \boldsymbol{\tau}_n$, giving a recurrence relation as stated in Theorem \ref{thm:sssrecurrence}.

The proof is divided into three parts. 
In the first part, we express $\lambda (\triangle_{n+1} - \boldsymbol{\tau}_{n+1})$ as a disjoint union of continuous parallel slices.
In the second part, we prove which slices are non-empty, and ``floss out'' the the slices without a lattice point.
In the last part, we prove that each non-empty slices of $\lambda (\triangle_{n+1} - \boldsymbol{\tau}_{n+1})$ are in bijective correspondence with the lattice points in some rational dilate of $\triangle_n - \boldsymbol{\tau}_n$

\subsubsection{Part $1$}

Let $\triangle_{n}' = \proj_{(x_{n} = 0)}(\triangle_{n})$, and $\boldsymbol{\tau}_n' = ({\tau}_1, \ldots, {\tau}_{n-1}, 0)$ from $\boldsymbol{\tau}_n$.
Obtain $\hat{\triangle}_n$ by lifting $\triangle_n$ to $\R^{n+1}$ by appending $0$ to the end of each lattice point of $\triangle_n$. 
Now we prove the following:

    \[
      \lambda(\triangle_{n+1}' - \boldsymbol{\tau}_{n+1}') \cap \Z^{n+1} = \bigcup_{0 \leq k \leq \lambda} \left(\lambda(\triangle_{n+1}' - \boldsymbol{\tau}_{n+1}') \cap (\aff(\lambda(\hat{\triangle}_n - \boldsymbol{\tau}_{n+1}')) + k\mathbf{v}) \cap \Z^{n+1}\right)
    \] 
    where $\mathbf{v} = (1, \ldots , 1, 0) \in \R^{n+1}$.

Let $\mathbf{0_n}$ be the origin in $\R^n$, $d$ be the distance between $\mathbf{0_{n+1}}$ and $\aff(\hat{\triangle}_n - \boldsymbol{\tau}_{n+1}')$, and $\mathbf{v}$ be any fixed vector on the plane $x_{n+1} = 0$ that is orthogonal to $\aff(\hat{\triangle}_n - \boldsymbol{\tau}_{n+1}')$ such that $\mathbf{v}\cdot \boldsymbol{\tau}_{n+1}' \geq 0$ (see Figure \ref{fig:claim}, for a concrete visual example).

\begin{figure}[!ht]
\centering
\scalebox{0.8}{\begin{tikzpicture}[scale=0.75]
 \definecolor{polytopecolor}{rgb}{0.0, 0.72, 0.92}
 \definecolor{darkgreen}{RGB}{1,50,32}	
\draw[color=lightgray,step=.5cm,
dashed, opacity=0.5] (-4,-7) grid (7,11);
\draw[->] (-4,4) -- (7.5,4)
node[below right] {};
\draw[->] (2,-7) -- (2,11.5)
node[left] {};

\draw[line width=0pt, green!10, fill=green!15] (-4,-2) -- (2,4) -- (2,-7) -- (-4,-7) -- (-4,-2);


\draw[line width=1.5pt, red, dashed, domain=-4:3] plot(\x,{-1*\x-4)}) node[right] {$\aff(\frac{5}{2}(\hat\triangle_2-\boldsymbol{\tau}_{3}'))$};

\draw[line width=1.5pt, red, dashed, domain=-4:4.5] plot(\x,{-1*\x-2)}) node[right] {};

\draw[line width=1.5pt, red, dashed, domain=-4:7] plot(\x,{-1*\x)}) node[right] { };

\draw[line width=1.5pt, red, dashed, domain=-4:7] plot(\x,{-1*\x+2)}) node[right] {$\aff(\hat\triangle_2-\boldsymbol{\tau}_{3}')$};

\draw[line width=1.5pt, red, dashed, domain=-4:7] plot(\x,{-1*\x+4)}) node[right] {};

\draw[line width=1.5pt, red, dashed, domain=-4:7] plot(\x,{-1*\x+6)}) node[right] {};


\draw[ultra thin, fill=polytopecolor!40] (4,8) -- (6,10) -- (6,6) -- (4,8) node[label={[shift={(.8,-.5)}, text=blue] \normalsize $\triangle'_3$}] {};

\draw[ultra thin, fill=gray!30] (2,4) -- (2,-6) -- (-3,-1) -- (2,4);

\draw[ultra thin, fill=polytopecolor!40] (2,4) -- (2,0) -- (0,2) -- (2,4) ;


\draw[line width=3pt, orange] (4,8) -- (6,6) node[label={[shift={(-1,0)}, text=orange] \normalsize $\hat\triangle_2$}] {};

\draw[line width=3pt, orange] (1,3) -- (2,2) {};

\draw[line width=1.5pt, green, <-, domain=-4:2] plot(\x,{\x+2)}) node[right] {};

\draw[line width=1.5pt, green, ->] (2,4) -- (2,-7.1);

\draw[line width=3pt, orange] (0,2) -- (2,0) {}; 

\draw[line width=3pt, orange] (-1,1) -- (2,-2){};

\draw[line width=3pt, orange] (-2,0) -- (2,-4){};

\draw[line width=3pt, orange] (-3,-1) -- (2,-6){};


\draw (4,8)
node[fill, color=andrespink,circle, minimum size,label={[shift={(-.65,-.5)}]\small (1,2,0)}] {};

\draw (6,10)
node[fill, color=andrespink,circle, minimum size, label={[shift={(0.65,-.5)}]\small (2,3,0)}] {};

\draw (6,8)
node[fill, color=andrespink,circle, minimum size, label={[shift={(0.65,-.5)}]\small (2,2,0)}] {};

\draw (6,6)
node[fill, color=andrespink,circle, minimum size, label={[shift={(0.65,-.5)}]\small (2,1,0)}] {};

\draw (2,4)
node[fill, color=andrespink,circle, minimum size,label={[shift={(0.65,-.25)}]\small (0,0,0)}] {};

\draw (2,2)
node[fill, color=andrespink,circle, minimum size, label={[shift={(-.5,-.8)}, text=blue] \normalsize $\triangle'_3 - \boldsymbol{\tau}_{3}$}] {};

\draw (2,0)
node[fill, color=andrespink,circle, minimum size] {};

\draw (2,-2)
node[fill, color=andrespink,circle, minimum size] {};

\draw (2,-4)
node[fill, color=andrespink,circle, minimum size] {};

\draw (2,-6)
node[fill, color=andrespink,circle, minimum size, label={[shift={(0.65,-.75)}]\small (0,-5,0)}] {};

\draw (0,2)
node[fill, color=andrespink,circle, minimum size] {};

\draw (0,0)
node[fill, color=andrespink,circle, minimum size] {};

\draw (0,-2)
node[fill, color=andrespink,circle, minimum size, label={[shift={(-.3,-1)}]\normalsize $\frac{5}{2}\left(\triangle'_3-\boldsymbol{\tau}_{3}\right)$}] {};

\draw (0,-4)
node[fill, color=andrespink,circle, minimum size] {};

\draw (-2,0)
node[fill, color=andrespink,circle, minimum size] {};

\draw (-2,-2)
node[fill, color=andrespink,circle, minimum size] {};

\draw (-2,-4)
node[color=green, label={[darkgreen, shift={(0,-1)}]\large $\cone_\mathbf{0_3}(\hat{\triangle}_2-\boldsymbol{\tau}_{3}')$}] {};

\end{tikzpicture}}
\caption{This figure corresponds to the set-up of the proof of Theorem \ref{thm:sssrecurrence}, where $\lambda=\frac{5}{2}$, $n=2$, and distance $d=\sqrt{2}$.}
\label{fig:claim}
\end{figure}

We will show that the lattice points in $\lambda$-th dilate of $\triangle_{n+1}' - \boldsymbol{\tau}_{n+1}'$ corresponds to the lattice points in the union of slices of $\lambda(\triangle_{n+1}' - \boldsymbol{\tau}_{n+1}')$ obtained by intersecting the polytope with a infinite set of hyperplanes $\aff(\lambda(\hat{\triangle}_n - \boldsymbol{\tau}_{n+1}')) + k\mathbf{v}$ for all real numbers $k \in [0, \frac{\lambda d}{|\mathbf{v}|}]$.

Take $\mathcal{H}_\mathbf{0_{n+1}}(\lambda(\hat{\triangle}_n - \boldsymbol{\tau}_{n+1}'))$ as defined using Definition \ref{S_halfspace}.
Then 
    \begin{align}
      \notag \lambda(\triangle_{n+1}' - \boldsymbol{\tau}_{n+1}') \cap \Z^{n+1} &= \lambda(\triangle_{n+1}' - \boldsymbol{\tau}_{n+1}') \cap \mathcal{H}_{\mathbf{0_{n+1}}}(\lambda(\hat{\triangle}_n - \boldsymbol{\tau}_{n+1}')) \cap \Z^{n+1}\\
      \notag &= \lambda(\triangle_{n+1}' - \boldsymbol{\tau}_{n+1}') \cap \left( \bigcup_{0 \leq k \leq \frac{\lambda d}{|\mathbf{v}|}} (\aff(\lambda(\hat{\triangle}_n - \boldsymbol{\tau}_{n+1}')) + k\mathbf{v}) \right) \cap \Z^{n+1}\\
      &= \bigcup_{0 \leq k \leq \frac{\lambda d}{|\mathbf{v}|}} \left(\lambda(\triangle_{n+1}' - \boldsymbol{\tau}_{n+1}') \cap (\aff(\lambda(\hat{\triangle}_n - \boldsymbol{\tau}_{n+1}')) + k\mathbf{v}) \cap \Z^{n+1}\right). \label{derived_equality}
    \end{align}
    Since,
    \[
      \aff(\hat{\triangle}_n - \boldsymbol{\tau}_{n+1}') = \lbrace\bx \in \R^{n+1} | x_1 + \cdots + x_n = -n, x_{n+1} = 0\rbrace,
    \]
   note that $\mathbf{v} = c(1, \ldots, 1, 0) \in \R^{n+1}$ for some $c \in \R$. 

 Without loss of generality, take $c=1$ and consider $\mathbf{v} = (1, \ldots , 1, 0) \in \R^{n+1}$. Note that $d = |\mathbf{v}|$ and thus, we can proceed with (\ref{derived_equality}) as follows:
    \[
      \lambda(\triangle_{n+1}' - \boldsymbol{\tau}_{n+1}') \cap \Z^{n+1} = \bigcup_{0 \leq k \leq \lambda} \left(\lambda(\triangle_{n+1}' - \boldsymbol{\tau}_{n+1}') \cap (\aff(\lambda(\hat{\triangle}_n - \boldsymbol{\tau}_{n+1}')) + k\mathbf{v}) \cap \Z^{n+1}\right)
    \] 
    with $\mathbf{v} = (1, \ldots , 1, 0) \in \R^{n+1}$.

\subsubsection{Part $2$}

We proceed by showing that lattice points can only exist on the slices of $\lambda(\triangle_{n+1}' - \boldsymbol{\tau}_{n+1}')$ obtained from the hyperplanes $\aff(\lambda(\hat{\triangle}_n - \boldsymbol{\tau}_{n+1}')) + k\mathbf{v}$ for $k \in [0, \lambda]$ such that $kn$ is a integer. 
Equivalently, we show that if $kn \notin [0, \lambda] \cap \Z$, then $\lambda(\triangle_{n+1}' - \boldsymbol{\tau}_{n+1}') \cap (\aff(\lambda(\hat{\triangle}_n - \boldsymbol{\tau}_{n+1}')) + k\mathbf{v}) = \emptyset$.
So, first we consider $\lambda(\triangle_{n+1} - \boldsymbol{\tau}_{n+1})$, which is $\lambda(\triangle_{n+1}' - \boldsymbol{\tau}_{n+1}')$ before the projection to $x_{n+1} = 0$.
Notice that the slice of $\lambda(\triangle_{n+1} - \boldsymbol{\tau}_{n+1})$ obtained from intersecting it with the hyperplane $H_h:=\lbrace\bx \in \R^{n+1} | x_{n+1} = h\rbrace$ only contains lattice points if $h$ is an integer. 
Thus,
    \[
      h \not\in \left(\Z \cap [-\lambda n, 0] \right) \Longrightarrow (\lambda(\triangle_{n+1} - \boldsymbol{\tau}_{n+1}) \cap \Z^{n+1}) \cap H_h = \emptyset.
    \]
    
Next, we show that the slice of $\lambda(\triangle_{n+1} - \boldsymbol{\tau}_{n+1})$ obtained from intersecting it with $H_h$ is equivalent to the slice of $\lambda(\triangle_{n+1} - \boldsymbol{\tau}_{n+1})$ obtained from intersecting it with $H^{\text{vert}} + (1 - \frac{h}{n}) \mathbf{v}$, where \[H^{\text{vert}} := \lbrace\bx \in \R^{n+1}\; |\; \sum_{i=1}^n x_i = -n\rbrace.\]
Note that \[\aff(\lambda(\triangle_{n+1} - \boldsymbol{\tau}_{n+1})) = \lbrace\bx \in \R^{n+1} \; |\; \sum_{i=1}^{n+1} x_i = 0\rbrace.\]
Now, we can derive the following: $H_h \cap \aff(\lambda(\triangle_{n+1} - \boldsymbol{\tau}_{n+1}))$ equals

    \begin{align*}
&\lbrace \bx \in \R^{n+1} \; |\; x_{n+1} = h\rbrace \cap \lbrace\bx \in \R^{n+1} \; |\; \sum_{i=1}^{n+1} x_i = 0\rbrace\\
      &= \lbrace\bx \in \R^{n+1} \; |\; \sum_{i=1}^{n} x_i = -h, x_{n+1} = h\rbrace\\
      &= \lbrace\bx \in \R^{n+1} \; |\; \sum_{i=1}^{n} x_i = -h\rbrace \cap \lbrace\bx \in \R^{n+1} \; |\; \sum_{i=1}^{n+1} x_i = 0\rbrace\\
      &= \left\lbrace\left(x_1 + \left(1-\frac{h}{n}\right), \ldots, x_n + \left(1-\frac{h}{n}\right), x_{n+1}\right) \;\Bigg|\; \sum_{i=1}^{n} x_i = -n \right\rbrace \cap \lbrace\bx \in \R^{n+1} \; |\; \sum_{i=1}^{n+1} x_i = 0\rbrace\\
      &= \left(\lbrace\bx \in \R^{n+1} \; |\; \sum_{i=1}^{n} x_i = -n\rbrace + \left(1-\frac{h}{n}\right)\mathbf{v}\right) \cap \aff(\lambda(\triangle_{n+1} - \boldsymbol{\tau}_{n+1}))\\
      &= \left(H^{\text{vert}} + \left(1-\frac{h}{n}\right)\mathbf{v}\right) \cap \aff(\lambda(\triangle_{n+1} - \boldsymbol{\tau}_{n+1})).
    \end{align*}  
    Hence,
    \[
      (\lambda(\triangle_{n+1} - \boldsymbol{\tau}_{n+1}) \cap \Z^{n+1}) \cap H_h = (\lambda(\triangle_{n+1} - \boldsymbol{\tau}_{n+1}) \cap \Z^{n+1}) \cap \left(H^{\text{vert}} + \left(1-\frac{h}{n}\right)\mathbf{v}\right).
    \]
    
Further, note that the slice of $\lambda(\triangle_{n+1} - \boldsymbol{\tau}_{n+1})$ obtained from intersecting it with $H^{\text{vert}} + (1-\frac{h}{n})\mathbf{v}$ has no lattice point on it if $h$ is not an integer, i.e.,
\[
h \not\in (\Z \cap [-\lambda n, 0]) \Longrightarrow (\lambda(\triangle_{n+1} - \boldsymbol{\tau}_{n+1}) \cap \Z^{n+1}) \cap \left(H^{\text{vert}} + \left(1-\frac{h}{n}\right)\mathbf{v}\right) = \emptyset.   
\]

Lastly, we show that the projection onto $x_{n+1}=0$ of the slice of $\lambda(\triangle_{n+1} - \boldsymbol{\tau}_{n+1})$ obtained from intersecting it with $H^{\text{vert}} + (1-\frac{h}{n})\mathbf{v}$ is equivalent to the slice of $\lambda(\hat{\triangle}_n - \boldsymbol{\tau}_{n+1}')$ obtained from intersecting it with $\aff(\lambda(\hat{\triangle}_n - \boldsymbol{\tau}_{n+1}')) + (1-\frac{h}{n})\mathbf{v}$.
So, consider
\[\proj_{x_{n+1} = 0}\left((\lambda(\triangle_{n+1} - \boldsymbol{\tau}_{n+1}) \cap \Z^{n+1}) \cap \left(H^{\text{vert}} + \left(1-\frac{h}{n}\right)\mathbf{v}\right)\right).\] 

Then
  \begin{align*}
    \proj_{x_{n+1} = 0}&\left((\lambda(\triangle_{n+1} - \boldsymbol{\tau}_{n+1}) \cap \Z^{n+1}) \cap \left(H^{\text{vert}} + \left(1-\frac{h}{n}\right)\mathbf{v}\right)\right)\\
      &= \proj_{x_{n+1} = 0}\left(\lambda(\triangle_{n+1} - \boldsymbol{\tau}_{n+1}) \cap \Z^{n+1}\right) \cap \proj_{x_{n+1} = 0}\left(H^{\text{vert}} + \left(1-\frac{h}{n}\right)\mathbf{v}\right)\\
      &= (\lambda(\triangle_{n+1}' - \boldsymbol{\tau}_{n+1}') \cap \Z^{n+1}) \cap \left(\lbrace\bx \in \R^{n+1} | x_1 + \cdots + x_n = -n, x_{n+1} = 0\rbrace + \left(1-\frac{h}{n}\right)\mathbf{v}\right)\\
      &= (\lambda(\triangle_{n+1}' - \boldsymbol{\tau}_{n+1}') \cap \Z^{n+1}) \cap \left(\aff(\lambda(\hat{\triangle}_n - \boldsymbol{\tau}_{n+1}')) + \left(1-\frac{h}{n}\right)\mathbf{v}\right).
    \end{align*}  
    Thus,
    \[
      h \not\in \Z \cap [-\lambda n, 0] \Longrightarrow (\lambda(\triangle_{n+1}' - \boldsymbol{\tau}_{n+1}') \cap \Z^{n+1}) \cap \left(\aff(\lambda(\hat{\triangle}_n - \boldsymbol{\tau}_{n+1}')) + \left(1-\frac{h}{n}\right)\mathbf{v}\right) = \emptyset.
    \]
    Therefore,
    \[
        kn \not\in \Z \cap [0, \lambda n] \Longrightarrow (\lambda(\triangle_{n+1}' - \boldsymbol{\tau}_{n+1}') \cap \Z^{n+1}) \cap \left(\aff(\lambda(\hat{\triangle}_n - \boldsymbol{\tau}_{n+1}')) + k\mathbf{v}\right) = \emptyset.
    \]

    To conclude, by \textit{``flossing''} out all the slices of $\lambda(\triangle_{n+1}' - \boldsymbol{\tau}_{n+1}')$ that do not contain lattice points, we can derive the following:
    \begin{align*}
        \bigcup_{0 \leq k \leq \lambda} &\left(\lambda(\triangle_{n+1}' - \boldsymbol{\tau}_{n+1}') \cap (\aff(\lambda(\hat{\triangle}_n - \boldsymbol{\tau}_{n+1}')) + k\mathbf{v}) \cap \Z^{n+1}\right)\\
      &= \biguplus_{\substack{0 \leq k \leq \lambda: \\ kn \in \Z}} \left(\lambda(\triangle_{n+1}' - \boldsymbol{\tau}_{n+1}') \cap (\aff(\lambda(\hat{\triangle}_n - \boldsymbol{\tau}_{n+1}')) + k\mathbf{v}) \cap \Z^{n+1}\right)\\
      &= \biguplus_{k = 0}^{\lfloor \lambda n \rfloor} \left(\lambda(\triangle_{n+1}' - \boldsymbol{\tau}_{n+1}') \cap (\aff(\lambda(\hat{\triangle}_n - \boldsymbol{\tau}_{n+1}')) + \frac{k}{n}\mathbf{v}) \cap \Z^{n+1}\right)\\
      &= \biguplus_{k = 0}^{\lfloor \lambda n \rfloor} \left(\cone_{\mathbf{0_{n+1}}}(\hat{\triangle}_n - \boldsymbol{\tau}_{n+1}') \cap (\aff(\lambda(\hat{\triangle}_n - \boldsymbol{\tau}_{n+1}')) + \frac{k}{n}\mathbf{v}) \cap \Z^{n+1}\right)\\
      &= \biguplus_{k = 0}^{\lfloor \lambda n \rfloor} \left(\frac{k}{n}(\hat{\triangle}_n - \boldsymbol{\tau}_{n+1}') \cap \Z^{n+1}\right).
    \end{align*}
    
Combining the result with what we have derived in Part $1$, we obtain the following:

For all $\lambda \in \R_{\geq 0}$,
    \begin{align*}
      \lambda(\triangle_{n+1}' - \boldsymbol{\tau}_{n+1}') \cap \Z^{n+1} = \biguplus_{k = 0}^{\lfloor \lambda n \rfloor} \left(\frac{k}{n}(\hat{\triangle}_n - \boldsymbol{\tau}_{n+1}') \cap \Z^{n+1}\right).
    \end{align*}

\subsubsection{Part $3$}
Finally, we use the fact that there is a bijective correspondence between $$\left(\frac{k}{n}(\hat{\triangle}_n - \boldsymbol{\tau}_{n+1}') \cap \Z^{n+1}\right) \text{ and } \left(\frac{k}{n}(\triangle_n-\boldsymbol{\tau}_n) \cap \Z^n\right).$$

\noindent Note that by Lemma \ref{lem5.5}, 
    \begin{equation}\label{eq:proof1}
        \left|\lambda(\triangle_{n+1} - \boldsymbol{\tau}_{n+1}) \cap \Z^{n+1}\right| = \left|\lambda(\triangle_{n+1}' - \boldsymbol{\tau}_{n+1}') \cap \Z^{n+1}\right|.
    \end{equation}
From what we have obtained in Part 2,
    \begin{align*}
      \left|\lambda(\triangle_{n+1}' - \boldsymbol{\tau}_{n+1}') \cap \Z^{n+1}\right| &= \left|\biguplus_{k = 0}^{\lfloor \lambda n \rfloor} \left(\frac{k}{n}(\hat{\triangle}_n - \boldsymbol{\tau}_{n+1}') \cap \Z^{n+1}\right)\right|
      = \sum_{k = 0}^{\lfloor \lambda n \rfloor}  \left|\frac{k}{n}(\hat{\triangle}_n - \boldsymbol{\tau}_{n+1}') \cap \Z^{n+1}\right|.
  \end{align*}
  Further, we have that
  \begin{align*}
      \left|\lambda(\triangle_{n+1} - \boldsymbol{\tau}_{n+1}) \cap \Z^{n+1}\right| &= \sum_{k=0}^{\lfloor n \lambda \rfloor}\left|\frac{k}{n}(\triangle_n-\overline{\boldsymbol{\tau}_{n+1}}) \cap \Z^n\right|
      = \sum_{k=0}^{\lfloor n \lambda \rfloor}\left|\frac{k}{n}(\triangle_n-\boldsymbol{\tau}_n) \cap \Z^n\right|,
    \end{align*}
where the second equality follows from Lemma \ref{lem:realtranslation}.

\subsection{Proof of Theorem \ref{thm:gorenstein}}\label{subsec:pfthm5.11}
First, we prove the claim: for all $t \in \Z$,
      \[
        \left|\left(\left(\frac{t}{n} + \frac{2n-1}{n}\right)(\triangle_n - \boldsymbol{\tau}_{n})\right)^{\circ} \cap \Z^n\right| = \left|\frac{t}{n}(\triangle_n - \boldsymbol{\tau}_n) \cap \Z^n\right|.
      \]
      In words, the relative interior lattice-point count for the $\left(\frac{t}{n} + \frac{2n-1}{n}\right)$-dilate of $\triangle_n - \boldsymbol{\tau}_{n}$ matches the lattice-point count of the $\frac{t}{n}$-dilate of $\triangle_n - \boldsymbol{\tau}_n$.  
      We proceed by induction on $n$, which relies on manipulations of summations. 
      
      Consider $n=2$.
      Observe that $\triangle_2-\boldsymbol{\tau}_2$ is simply a line segment from $(0,0)$ to $(-1,1)$, $(\frac{t}{2} + \frac{3}{2})(\triangle_2-\boldsymbol{\tau}_2)$ is a line segment from $(0,0)$ to $(-\frac{t}{2} - \frac{3}{2}, \frac{t}{2} + \frac{3}{2})$, and $\frac{t}{2}(\triangle_2-\boldsymbol{\tau}_2)$ is a line segment from $(0,0)$ to $(-\frac{t}{2}, \frac{t}{2})$.
      Notice that all the interior lattice points of the line segment from $(0,0)$ to $(-\frac{t}{2} - \frac{3}{2}, \frac{t}{2} + \frac{3}{2})$ lie on the line segment from $(0,0)$ to $(-\frac{t}{2}, \frac{t}{2})$ giving the following relation:
      \[
        L_\R\left((\triangle_2 - \boldsymbol{\tau}_2)^\circ; \frac{t}{2} + \frac{3}{2}\right) = L_\R\left(\triangle_2 - \boldsymbol{\tau}_2; \frac{t}{2}\right)
      \]
    for all $t \in \Z$.

Assume the statement holds for some $m \in \Z_{\geq 2}$.
We then have:
\begin{align*}
        \left|\frac{t}{m+1}(\triangle_{m+1} - \boldsymbol{\tau}_{m+1}) \cap \Z^{m+1}\right| &= \sum_{k=0}^{\lfloor m \cdot \frac{t}{m+1} \rfloor} \left|\frac{k}{m}(\triangle_m - \boldsymbol{\tau}_m) \cap \Z^m\right| \\
         &= \underbrace{\sum_{k=0}^{\lfloor m \cdot \frac{t}{m+1}\rfloor} \left|\left(\frac{k}{m}+\frac{2m-1}{m}(\triangle_m - \boldsymbol{\tau}_m)\right)^{\circ} \cap \Z^m\right|}_{(\bigstar)}.
      \end{align*}
      By Claim \ref{claim: floorfunctions}, we can derive the following:
      \begin{align*}
        (\bigstar) &= \sum_{k=2m-1}^{\lceil m \cdot \frac{t}{m+1} + \frac{2m+1}{m+1}\rceil -1} \left|\left(\frac{k}{m}(\triangle_m - \boldsymbol{\tau}_m)\right)^{\circ} \cap \Z^m \right|\\
        &= \sum_{k=0}^{\lceil m \cdot \frac{t}{m+1} + \frac{2m+1}{m+1}\rceil -1} \left|\left(\frac{k}{m}(\triangle_m - \boldsymbol{\tau}_m)\right)^{\circ} \cap \Z^m \right| -  \sum_{k=0}^{2m-2} \left|\left(\frac{k}{m}(\triangle_m - \boldsymbol{\tau}_m)\right)^{\circ} \cap \Z^m \right|\\
        &= \left|\left(\left(\frac{t}{m+1} + \frac{2m+1}{m+1}\right)(\triangle_{m+1} - \boldsymbol{\tau}_{m+1})\right)^{\circ} \cap \Z^{m+1}\right| - \sum_{k=0}^{2m-2} \left|\left(\frac{k}{m}(\triangle_m - \boldsymbol{\tau}_m)\right)^{\circ} \cap \Z^m \right|.
      \end{align*}
      Here, note that by Proposition \ref{prop:interiorrecurrence},
      \begin{align*}
        \sum_{k=0}^{2m-2} \left|\left(\frac{k}{m}(\triangle_m - \boldsymbol{\tau}_m)\right)^{\circ} \cap \Z^m \right| &=  \left|\left(\frac{2m-1}{m}(\triangle_{m+1} - \boldsymbol{\tau}_{m+1})\right)^{\circ} \cap \Z^{m+1}\right|\\
        &= \left|\left(\triangle_{m+1} - \boldsymbol{\tau}_{m+1}\right)^{\circ} \cap \Z^{m+1}\right|,\\
        &= 0,
      \end{align*}
      where the last equality holds by Theorem \ref{thm:hollow}. 
      Hence,
      \[
       (\bigstar) = \left|\left(\left(\frac{t}{m+1} + \frac{2m+1}{m+1}\right)(\triangle_{m+1} - \boldsymbol{\tau}_{m+1})\right)^{\circ} \cap \Z^{m+1}\right|.
      \]
      Therefore, by induction, we have proved that for all $t \in \Z_{\geq 0}$, $n \in \Z_{\geq 2}$,
      \[
        \left|\left(\left(\frac{t}{n} + \frac{2n-1}{n}\right)(\triangle_n - \boldsymbol{\tau}_{n})\right)^{\circ} \cap \Z^n\right| = \left|\frac{t}{n}(\triangle_n - \boldsymbol{\tau}_n) \cap \Z^n\right|.
      \]

      Now, by Theorem \ref{thm:sssrecurrence}, for all $n \in \Z_{\geq 3}$
      \begin{align*}
        L_\R(\triangle_n - \boldsymbol{\tau}_n ; \lambda) &= \sum_{k=0}^{\lfloor(n-1)\lambda\rfloor} L_\R\left(\triangle_{n-1}-\boldsymbol{\tau}_{n-1}; \frac{k}{n}\right).
        \intertext{By our recently proven claim,}
        &= \sum_{k=0}^{\lfloor(n-1)\lambda\rfloor} \left|\left(\left(\frac{k}{n-1} + \frac{2n-3}{n-1}\right)(\triangle_{n-1} - \boldsymbol{\tau}_{n-1})\right)^{\circ} \cap \Z^{n-1}\right|\\
        &= \sum_{k = 0}^{\lceil(n-1)(\lambda+2)\rceil -1} \left|\left(\left(\frac{k}{n-1}\right)(\triangle_{n-1} - \boldsymbol{\tau}_{n-1})\right)^{\circ} \cap \Z^{n-1}\right|\\
        &\text{\;\;\;\;\;\;\;\;\;\;\;\;\;\;\;\;\;\;} - \sum_{k = 0}^{2n-4} \left|\left(\left(\frac{k}{n-1}\right)(\triangle_{n-1} - \boldsymbol{\tau}_{n-1})\right)^{\circ} \cap \Z^{n-1}\right|,\\
        \intertext{and by Theorem \ref{thm:hollow},}
        &= \left|((\lambda+2)(\triangle_n - \boldsymbol{\tau}_n))^{\circ} \cap \Z^{n}\right| - 0\\
        &= L_\R((\triangle_n - \boldsymbol{\tau}_n)^{\circ}; \lambda+2).
      \end{align*}
      One can also show by solving inequalities that for $n = 2$ the result holds, but we omit the technicalities here. 
      Therefore, for all $n \in \Z_{\geq 2}$,
      \[
        L_\R(\triangle_n - \boldsymbol{\tau}_n ; \lambda) = L_\R((\triangle_n - \boldsymbol{\tau}_n)^{\circ}; \lambda+2).
      \]
      In particular, for all $n \in \Z_{\geq 2}$ and $t \in \Z_{\geq 0}$, $\triangle_n$ is Gorenstein of index $2$, i.e., 
      \[
        L_\Z(\triangle_n ; t) =  L_\Z(\triangle_n^{\circ} ; t+2).
      \]

\vspace{0.25cm}

\section*{Dedication \& Acknowledgments}
\begin{center}
In loving memory of a dear friend, Luca Elghanayan.   
\end{center}
\vspace{0.25cm} 
The authors thank the anonymous referee for pointing out relevant literature. 
The authors are grateful to Matthias Beck, Simon Rubenstein-Salzedo, and Julie Vega for fruitful conversation.
Andr\'es R.~Vindas-Mel\'endez also thanks Cameron Lowe, Justin McClung, Emily Muller-Foster, and Heather Willis for conversations at the start of this project. Eon Lee is supported by Basic Science Research Program through the National Research Foundation of Korea (NRF) funded by the Ministry of Education (NRF-2022R1F1A1063424). Andr\'es R.~Vindas-Mel\'endez was partially supported by the National Science Foundation under Award DMS-$2102921$.

\bibliographystyle{amsplain}
\bibliography{references}

\section{Appendix}\label{appendix}

\begin{claim}\label{claim: floorfunctions}
    For $\boldsymbol{\tau}_n$, we have
    \[
        \sum_{k=0}^{\lfloor n \cdot \frac{t}{n+1}\rfloor} \left|\left(\frac{k}{n}+\frac{2n-1}{n}(\triangle_n - \boldsymbol{\tau}_n)\right)^{\circ} \cap \Z^n\right| = \sum_{k=2n-1}^{\lceil n \cdot \frac{t}{n+1} + \frac{2n+1}{n+1}\rceil -1} \left|\left(\frac{k}{n}(\triangle_n - \boldsymbol{\tau}_n)\right)^{\circ} \cap \Z^n \right|.
    \]
\end{claim}
\begin{proof}
    By dividing the cases for when $n+1 | t$ and $n+1 \nmid t$, we can derive the following:
    \begin{align*}
        \sum_{k=0}^{\lfloor n \cdot \frac{t}{n+1}\rfloor} &\left|\left(\frac{k}{n}+\frac{2n-1}{n}(\triangle_n - \boldsymbol{\tau}_n)\right)^{\circ} \cap \Z^n\right| \\
        &=
        \begin{cases}
          \displaystyle \sum_{k=2n-1}^{\lceil n \cdot \frac{t}{n+1} + 2n \rceil -1} \left|\left(\frac{k}{n}(\triangle_n - \boldsymbol{\tau}_n)\right)^{\circ} \cap \Z^n \right| & \text{if } n+1 | t\\
          \displaystyle \sum_{k=2n-1}^{\lceil n \cdot \frac{t}{n+1} + 2n -1\rceil -1} \left|\left(\frac{k}{n}(\triangle_n - \boldsymbol{\tau}_n)\right)^{\circ} \cap \Z^n \right| & \text{if } n+1 \nmid t.
        \end{cases}
      \end{align*}
      Note that if $n+1 \mid t$, then $n \cdot \frac{t}{n+1} + 2n \in \Z$, thus
      \[
      \left\lceil n \cdot \frac{t}{n+1} + 2n \right\rceil = \left\lceil n \cdot \frac{t}{n+1} + 2n - \frac{n}{n+1} \right\rceil.
      \]
      Alternatively, if $n+1 \nmid t$, then $0 < n \cdot \frac{t}{n+1} + 2n -1 - \lceil n \cdot \frac{t}{n+1} + 2n -1 \rceil \leq \frac{n}{n+1}$, thus
      \[
      \left\lceil n \cdot \frac{t}{n+1} + 2n -1 \right\rceil = \left\lceil n \cdot \frac{t}{n+1} + 2n -1 + \frac{1}{n+1}\right\rceil.
      \]
      Therefore,
      \[
        \sum_{k=0}^{\lfloor n \cdot \frac{t}{n+1}\rfloor} \left|\left(\frac{k}{n}+\frac{2n-1}{n}(\triangle_n - \boldsymbol{\tau}_n)\right)^{\circ} \cap \Z^n\right| = \sum_{k=2n-1}^{\lceil n \cdot \frac{t}{n+1} + \frac{2n+1}{n+1}\rceil -1} \left|\left(\frac{k}{n}(\triangle_n - \boldsymbol{\tau}_n)\right)^{\circ} \cap \Z^n \right|.
      \]
\end{proof}

\end{document}